\documentclass[final,leqno]{siamltex}
\usepackage{amsmath,amssymb}
\usepackage{graphicx}
\usepackage{verbatim}

\newtheorem{remark}[theorem]{Remark}
\newtheorem{conjecture}[theorem]{Conjecture}

\newcommand{\csta}{C_{\rm sta}}

\newcommand{\ep}{\varepsilon}

\newcommand{\om}{\omega}
\newcommand{\Om}{\Omega}
\newcommand{\pa}{\partial}

\renewcommand{\i}{{\rm\mathbf i}}

\DeclareMathOperator{\re}{{Re}}
\DeclareMathOperator{\im}{{Im}}

\newcommand{\bL}{\mathbf{L}}
\newcommand{\bH}{\mathbf{H}}

\newcommand{\bV}{\mathbf{V}}

\newcommand{\p}{\partial}

\newcommand{\Ome}{\Omega}

\newcommand{\ddiv}{\mbox{\rm div\,}}

\newcommand{\bff}{\mathbf{f}}

\newcommand{\bfe}{\mathbf{e}}
\newcommand{\bfn}{\mathbf{n}}

\newcommand{\bfg}{\mathbf{g}}

\newcommand{\bfu}{\mathbf{u}}

\newcommand{\bfv}{\mathbf{v}}
\newcommand{\bfw}{\mathbf{w}}
\newcommand{\bfx}{\mathbf{x}}

\newcommand{\bfE}{\mathbf{E}}

\newcommand{\bfH}{\mathbf{H}}
\newcommand{\bfF}{\mathbf{F}}

\newcommand{\bfG}{\mathbf{G}}
\newcommand{\bfP}{\mathbf{P}}

\newcommand{\bfV}{\mathbf{V}}
\newcommand{\bfU}{\mathbf{U}}

\newcommand{\bfi}{\mathbf{i}}
\newcommand{\bfa}{\boldsymbol{ \alpha}}
\newcommand{\bnabla}{\boldsymbol{\nabla}}
\newcommand{\bdiv}{\mbox{\rm \bf div\,}}

\newcommand{\bftu}{\tilde{\mathbf{u}}}
\newcommand{\bfhu}{\hat{\mathbf{u}}}

\newcommand{\bfhw}{\hat{\mathbf{w}}}
\newcommand{\bfpsi}{\boldsymbol{ \psi}}
\newcommand{\bfphi}{\boldsymbol{ \phi}}

\title{An unconditionally stable discontinuous Galerkin method for the elastic Helmholtz 
equations with large frequency}
\author{
Xiaobing Feng\thanks{Department of Mathematics, The University of
Tennessee, Knoxville, TN 37996, U.S.A.  ({\tt xfeng@math.utk.edu}).
The work of this author was partially supported by the NSF grants DMS-1016173 and DMS-1318486.}
\and
Cody Lorton\thanks{Department of Mathematics and Statistics, University of
West Florida, Pensacola, FL 32514, U.S.A.  ({\tt clorton@uwf.edu}).
The work of this author was partially supported by the NSF grants DMS-1016173 and DMS-1318486.}
}

\begin{document}

\maketitle

%\large
\begin{abstract}
In this paper we propose and analyze an interior penalty discontinuous Galerkin
(IP-DG) method using piecewise linear polynomials for the elastic Helmholtz equations 
with the first order absorbing boundary condition. It is proved that the sesquilinear form
for the problem satisfies a generalized weak coercivity property, which immediately infers
a stability estimate for the solution of the differential problem in all frequency regimes.
It is also proved that the proposed 
IP-DG method is unconditionally stable with respect to both frequency $\om$ and
mesh size $h$. Sub-optimal order (with respect to $h$) error estimates 
in the broken $H^1$-norm and in the $L^2$-norm are 
obtained in all mesh regimes. These estimate improve
to optimal order when the mesh size $h$ is restricted to the pre-asymptotic regime  
(i.e., $\om^{\beta} h =O(1)$ for some $1\leq \beta<2$).  The novelties of 
the proposed IP-DG method include:
first, the method penalizes not only the jumps of the function values
across the element edges but also the jumps of the normal derivatives; 
second, the penalty parameters are taken as complex numbers with positive imaginary parts.
In order to establish the desired unconditional stability 
estimate for the numerical solution, the main idea is to exploit a (simple) property 
of linear functions to overcome the main difficulty caused by non-Hermitian 
nature and strong indefiniteness of the Helmholtz-type problem. The error 
estimate is then derived using a nonstandard technique adapted from \cite{Feng_Wu09}.
Numerical experiments are also presented to validate the theoretical results and
to numerically examine the pollution effect (with respect to $\om$) in the error bounds. 
\end{abstract}



%%%%%%%%%%%%%%
\section{Introduction}\label{sec-1}
Wave phenomena are pervasive throughout many scientific fields, wave computation 
has been one of the central topics in computational sciences and has also 
been at the forefront of scientific computing. It becomes more and 
more important as the boundary of wave-related application problems keeps pushing outward. Solving these application problems largely hinges on
computing the solutions of the governing wave equation(s).  
Among them those involving high frequency waves are often most difficult 
to analyze and solve numerically.  This is due to the fact that a high 
frequency (or large wave number), which means a very short wave length, 
results in a strongly indefinite PDE problem and a highly oscillatory solution. 
Consequently, very fine meshes must be used to resolve the wave, 
which in turn results in extremely large non-Hermitian and strongly 
indefinite algebraic systems to solve, a daunting task especially in high dimensions.

In this paper we shall focus our attention on one type of wave phenomena, that is wave 
propagation through some isotropic elastic media. This type is often encountered
in applications from materials science, medical science, petroleum engineering, 
just to name a few. Specifically, we shall consider the following elastic Helmholtz problem: 
\begin{align}
	-\om^2 \rho \bfu - \bdiv(\sigma(\bfu)) &= \bff \qquad \mbox{ in } \Om, \label{Eq:ElasticPDE} \\
	\bfi \om A \bfu + \sigma(\bfu) \bfn &= \bfg \qquad \mbox{ on } \Gamma:=\p\Ome, 
\label{Eq:ElasticBC}
\end{align}
where $\Om \subset \mathbb{R}^d, d = 2,3$, which represents an isotropic elastic media, 
is an open and bounded domain. $\bfn$ denotes the unit outward normal vector to $\p\Ome$, 
and $\bfi = \sqrt{-1}$ denotes the imaginary unit. 
$\bfu : \Om \to \mathbb{C}^d$ denotes the (Fourier-transformed) displacement vector of the elastic 
media $\Om$. $\om$ and $\rho$ are both positive constants for which $\om$ denotes 
the frequency of the sought-after elastic wave traveling through $\Om$ with speed $\rho$.  
$A$ is a $d \times d$ real symmetric positive definite matrix. $\sigma(\bfu)$ is the 
stress-tensor in $\Om$ defined by
%%%%%%
\begin{align}
\sigma(\bfu) := 2\mu \varepsilon(\bfu) + \lambda \ddiv \bfu I, 
\qquad \varepsilon(\bfu):= \frac{1}{2}\left(\bnabla \bfu + \bnabla \bfu^T \right), \notag
\end{align}
where $\mu > 0$ and $\lambda > 0$ are the Lam\'{e} constants for the elastic media $\Om$.

Equations \eqref{Eq:ElasticPDE} and \eqref{Eq:ElasticBC} can arise either from the frequency domain 
treatment of linear elastic wave equations  
(cf. \cite{Douglas_Sheen_Santos94}) or from seeking a time-harmonic solution to the 
following linear elastic wave equations: 
\begin{alignat}{2}
\rho \bfU_{tt} - \bdiv(\sigma(\bfU)) & = \bfF 
&&\qquad \mbox{ in } \Om \times (0,\infty), \label{Eq:LinearElasticPDE} \\
A \bfU + \sigma(\bfU) \bfn &= \bfG 
&&\qquad \mbox{ on } \Gamma \times (0, \infty), \label{Eq:LinearElasticBC} \\
\bfU = \bfU_t &= \mathbf{0} &&\qquad \mbox{ in } \Om \times \{t = 0 \}. \label{Eq:LinearElasticIC}
\end{alignat}
When $\mathbf{G}=\mathbf{0}$ the boundary condition \eqref{Eq:LinearElasticBC} is known
as the first-order absorbing (radiation) boundary condition \cite{Enquist_Majda79}. 
This boundary condition is an artificial boundary condition which has the property of
completely absorbing incoming plane waves that are perpendicular to the boundary 
$\p\Omega$. This boundary condition is imposed on the boundary $\p\Omega$ of the truncated domain $\Omega$ of
some unbounded domain with the intent of making the problem computationally feasible.  

As is the case for the scalar Helmholtz equation, which governs time-harmonic acoustic waves, 
the frequency $\om$ also plays a key role in the analysis and implementation of any numerical 
method for the elastic Helmholtz equations.  It is a well-known fact that in order to 
resolve the wave numerically one must use some minimum number of grid points in each 
wave length $\ell = 2 \pi/\om$ in every coordinate direction.  This requires a minimum 
mesh constraint $\om h = O(1)$, where $h$ denotes the mesh size.  In fact, the widely held
``rule of thumb" is to use $6 - 10$ grid points per wave length.  
In \cite{Ihlenburg_Babuska95}, Bab\v{u}ska \emph{et al} proved the necessity of 
this ``rule of thumb" in the $1$-dimensional case of the scalar Helmholtz equation.  
In \cite{Ihlenburg_Babuska95} it was also shown that the H$^1$ error bound for the 
finite element solution contains a pollution term that contributes to the loss of 
stability for the finite element method in the case of a large frequency. This pollution term 
also causes the error to increase as $\om$ increases under the mesh constraint $\om h = O(1)$.  
This forces one to adopt a more stringent mesh condition to guarantee an accurate 
numerical solution.

The loss of stability of the standard finite element method applied to Helmholtz-type 
problems remains as an open issue to be resolved.  In 
\cite{Aziz_Kellogg79,Douglas_Santos_Sheen_Bennethum93,Douglas_Sheen_Santos94} 
a stringent mesh condition of $\om^2 h = O(1)$ (called the asymptotic mesh regime) 
was required to obtain optimal and quasi-optimal error estimates for the scalar 
Helmholtz equation. In \cite{Cummings01} a similar mesh constraint of the form 
$\om^{\beta} h = O(1)$ was used to obtain error estimates for the elastic 
Helmholtz equations, where $\beta \geq 2$ is a constant to be discussed later 
in this paper. Requiring such a stringent mesh constraint makes the use of a practical 
coarse approximation space impossible in the case that $\om$ is large.  This is a 
major hurdle that must be overcome if one wishes to use multi-level algebraic solvers 
such as the multigrid method or the multi-level domain decomposition method.

The primary goal of this paper is to develop and analyze an interior penalty 
discontinuous Galerkin approximation method which is unconditionally stable with respect to 
both $\om$ and $h$ for the above elastic Helmholtz problem, which is 
different from the goal of the earlier DG work \cite{RSWW03}.  Specifically, we seek 
a numerical method in which a priori solution (or stability ) estimates can be 
obtained for all $\om, h > 0$.  Our work in this paper can be regarded as the elastic counterpart
of those reported in \cite{Feng_Wu09,Wu12,Feng_Wu14}, where the scalar Helmholtz equation
and the time-harmonic Maxwell's equations were the focuses of study. We like to 
note that all three Helmholtz-type problems share the above mentioned two
main difficulties, namely, strong indefiniteness and non-Hermitian character.
However, due to the fact that to their leading operators are very different, in particular, 
they have quite different kernel spaces, the formulation and
analysis of the desired IP-DG methods are significantly different,
although they are conceptually similar.

The remainder of the paper is organized as follows.  In Section \ref{sec-1b} we establish
a new generalized weak coercivity property for the elastic operator $-\ddiv (\sigma(\cdot))$,
and derive a stability estimate for the PDE solution as a corollary of the generalized weak 
coercivity property.
In Section \ref{sec-2} we introduce some notation and present the formulation of our IP-DG method 
studied in the rest of the paper.  In Section \ref{sec-3} we derive a stability (or a priori solution)
estimate for the numerical solution in all mesh regimes including the pre-asymptotic regime 
(i.e. $\om^\beta h=O(1)$ for some $1\leq \beta <2$). In Section \ref{sec-4} we employ the non-standard 
error estimate technique, which uses an elliptic projection along with the (discrete) stability estimate, 
to establish optimal order error estimates in both broken H$^1$-norm and the $L^2$-norm.  
Section \ref{sec:elastic_IPDG_numerics} is devoted to numerical experiments that validate our
theoretical results and demonstrate the proposed method's advantage over the standard finite element method.

%%%%%%%%%%%%%%%%%%%%%%%%%%%%%%%%%%%%%%%%%%%%%%%%%%%%%%%%%%%%%%%%%%%%%%%%%%
\section{Generalized inf-sup condition and stability estimates for PDE solutions} \label{sec-1b}
The goal of this section is to prove that the sesquilinear form
$a(\cdot,\cdot)$ defined in the weak formulation of 
\eqref{Eq:ElasticPDE}--\eqref{Eq:ElasticBC} satisfies a generalized weak
coercivity property.  Similar results were also obtained for the scalar
Helmholtz equation and the time-harmonic Maxwell's equations 
\cite{Feng_Wu14, Lorton_14}.  As a direct consequence of the weak coercivity
property proved in this section, we will obtain stability estimates for the
solution $\bfu$ of \eqref{Eq:ElasticPDE}--\eqref{Eq:ElasticBC}.

Throughout the rest of this paper we adopt the standard $L^2$-space norm and
inner product notation.  In particular, for $Q \subset \Om$ and
$\Sigma \subset \Gamma$, let $(\cdot,\cdot)_{Q}$ and
$\langle \cdot, \cdot \rangle_{\Sigma}$ denote the $L^2$ inner product on 
the complex-valued inner product spaces $\bL^2(Q) = \big( L^2(Q) \big)^d$ and
$\bL^2(\Sigma) = \big(L^2(\Sigma) \big)^d$, respectively.

We will study the weak formulation of 
\eqref{Eq:ElasticPDE}--\eqref{Eq:ElasticBC} defined as seeking $\bfu \in 
\bH^1(\Om)$ such that 
\begin{align} \label{Eq:WeakForm1a}
	a(\bfu,\bfv) = (\bff, \bfv)_{\Om} + \langle \bfg, \bfv \rangle_{\Gamma}
		\qquad \forall \bfv \in \bH^1(\Om),
\end{align}
where $a(\cdot,\cdot)$ is the sesquilinear form defined on $\bH^1(\Om) \times 
\bH^1(\Om)$ as
\begin{align} \label{Eq:WeakForm1b}
	a(\bfu,\bfv) := \lambda \big( \ddiv \bfu, \ddiv \bfv \big)_{\Om}
		+ 2\mu \big( \ep(\bfu), \ep(\bfv) \big)_{\Om} - \om^2 \rho \big( 
		\bfu, \bfv \big)_{\Om} + \i \om \big \langle A \bfu, \bfv \big 
		\rangle_{\Gamma}.
\end{align}

To show that $a(\cdot,\cdot)$ satisfies a generalized weak coercivity
property we will need to assume that the domain $\Om$ is ``nice''.  In 
particular, we assume that $\Om$ is a strictly star-shaped domain, i.e. there 
exists $x_0 \in \Om$ and $c_0 > 0 $ such that 
\begin{align} \label{StarShape2}
	(\bfx - \bfx_0) \cdot \bfn \geq c_0 \qquad \forall x \in \Gamma.
\end{align}

We will need to make use of Korn's second inequality to achieve our result. 
This inequality is given in the following lemma.

\begin{lemma} \label{lem:KornInequality}
	There exists a positive constant $K$ such that for any $\bfv \in 
	\bH^1(\Om)$ the following inequality holds
	\begin{align} \label{KornInequality}
		\| \bfv \|_{H^1(\Om)} \leq K \Big[ \| \ep(\bfv) \|_{L^2(\Om)} + \| 
		\bfv \|_{L^2(\Om)} \Big].
	\end{align}
\end{lemma}

Lastly, we will need to assume that the solution $\bfu$ to 
\eqref{Eq:WeakForm1a} satisfies a Korn-type inequality on the boundary $
\Gamma$.  This boundary Korn inequality was conjectured in 
\cite{Cummings_Feng06} and it was necessary for the authors to obtain 
stability estimates for the Elastic Helmholtz equation that are optimal in 
the frequency $\om$.  Here it is stated as a conjecture.

\begin{conjecture}\label{conj}
	There exists a positive constant $\tilde{K}$ such that for any $\bfu \in 
	\bH^2(\Om)$, the following Korn-type inequality holds
	\begin{align} \label{BoundaryKorn}
		\| \bnabla \bfu \|^2_{L^2(\Gamma)} \leq \tilde{K} \Big[ \| \bfu \|
		^2_{L^2(\Gamma)} + \| \ep (\bfu) \|^2_{L^2(\Gamma)} \Big].
	\end{align}
\end{conjecture}

With this boundary Korn inequality in mind, we define the following function spaces.
\begin{align*}
	\bV &:= \Big \{ \bfv \in \bH^1(\Om); \,\, |\nabla \bfu|_\Gamma \in \bL^2(\Gamma) \Big \}, \\
	\bV_{\tilde{K}} &:= \Big\{ \bfv \in \bV; \,\, \mbox{ 
		\eqref{BoundaryKorn} holds with constant } \tilde{K} \Big \}.
\end{align*}

We also note that since $A$ in \eqref{Eq:ElasticBC} is SPD, there exist positive constants 
$c_A$ and $C_A$ such that
\begin{align} \label{SPDInequality}
	c_A \| \bfv \|^2_{L^2(\Gamma)} \leq \big \langle A \bfv, \bfv \big 
		\rangle_{\Gamma} \leq C_A \| \bfv \|^2_{L^2(\Gamma)} \qquad \forall 
		\bfv \in \bL^2(\Gamma).
\end{align}

The following theorem states the generalized weak coercivity property for the 
sesquilinear form $a(\cdot,\cdot)$.

\begin{theorem} \label{thm:GenWeakCoercivity}
	Let $\Om$ be a strictly star-shaped domain with $diam(\Om) = R$ and let $
	\tilde{K}$ be a positive	constant.  Then for any $\bfu \in \bV_{\tilde{K}}
	$ the following inequality holds
	\begin{align} \label{GenWeakCoercivity}
		\sup_{\bfv \in \bV} \frac{| \im a(\bfu, \bfv) |}{\| \bfv \|_E} + 
			\sup_{\bfv \in \bH^1(\Om)} \frac{| \re a(\bfu, \bfv) |}{||| \bfv 
			|||_{L^2(\Om)}} \geq \frac{1}{\gamma} \| \bfu \|_E,
	\end{align}
	where
	\begin{align*}
		\gamma &:= \left[ 4 R^2 K \left( 1 + \frac{\om^2 \rho}{2 \mu} \right) 
		 	+ 4R^2 \tilde{K} + (d-1)^2 \right]^{\frac{1}{2}}, \\
		 M &:= \frac{1}{c_A c_0 \mu} \left( R \om^2 \rho c_0 \mu + R^2 \om 
		 	C_A \tilde{K} + c_0^2 \mu^2 \right), \\
		 \| \bfv \|^2_E &:= \om^2 \rho \| \bfv \|^2_{L^2(\Om)} + \lambda \| 
		 	\ddiv \bfv \|^2_{L^2(\Om)} + 2 \mu \| \ep(\bfv) \|_{L^2(\Om)} \\
		 & \qquad + c_0 \lambda \| \ddiv \bfv \|^2_{L^2(\Gamma)} + c_0 \mu \| 
		 	\ep (\bfv) \|^2_{L^2(\Gamma)}, \\
		 ||| \bfv |||_{L^2(\Ome)}^2 &:= \om^2 \rho \| \bfv \|^2_{L^2(\Om)} + c_0 \mu \| 
		 	\bfv \|^2_{L^2(\Gamma)}.
	\end{align*}
\end{theorem}

\begin{proof}
In this proof we assume that $\bfu \in \bV_{\tilde{K}} \cap \bH^2(\Om)$, noting that once 
we obtain the result on this more restrictive space that we can obtain the desired result 
for $\bfu \in \bV_{\tilde{K}}$ by the standard density argument.

By letting $\bfv = \bfu$ in \eqref{Eq:WeakForm1b} and taking the real and 
imaginary parts separately we get
\begin{align}
	\re a(\bfu, \bfu) &= \lambda \| \ddiv \bfu \|^2_{L^2(\Om)} + 2 \mu \| \ep 
		(\bfu) \|^2_{L^2(\Om)} - \om^2 \rho \| \bfu \|^2_{L^2(\Om)}, 
		\label{Eq:GWCa} \\
	\im a(\bfu, \bfu) &= \om \big \langle A \bfu, \bfu \big \rangle_{\Gamma}. 
		\label{Eq:GWCb}
\end{align}

For the remainder of this proof we let $\bfv = (\bnabla \bfu) \bfa$, where 
$\bfa := \bfx - \bfx_0$ and $\bfx_0 \in \Om$ comes from the star-shape 
property on the domain given by \eqref{StarShape2}. With this choice of test 
function $\bfv$, we can use the following integral identities for the elastic 
Helmholtz operator \cite{Cummings_Feng06}.
\begin{align} \label{Eq:Rellich1}
	&\lambda \big \langle \bfa \cdot \bfn, | \ddiv \bfu |^2 \big  
		\rangle_{\Gamma} + 2 \mu \big \langle \bfa \cdot \bfn, | \ep(\bfu)|^2 
		\big \rangle_{\Gamma} \\
	& \qquad = (d - 2) \Big( \lambda \| \ddiv \bfu \|^2_{L^2(\Om)} + 2 \mu \| 
		\ep(\bfu) \|^2_{L^2(\Om)} \Big) \notag \\
	& \qquad \qquad + 2 \re \Big( \lambda \big( \ddiv \bfu, \ddiv \bfv 
		\big)_{\Om} + 2 \mu \big( \ep(\bfu), \ep(\bfv) \big)_{\Om} \Big), 
		\notag
\end{align}
and
\begin{align}
	& d \| \bfu \|^2_{L^2(\Om)} = \big \langle \bfa \cdot \bfn, |\bfu|^2 \big 
		\rangle_{\Gamma} - 2 \re (\bfu, \bfv)_{\Om}. \label{Eq:Rellich2}
\end{align}

With $\bfv = \big( \bnabla \bfu \big) \bfa$ in \eqref{Eq:WeakForm1b}, we 
apply the above identities to get
\begin{align*}
	2 \re a( \bfu, \bfv) &= \lambda \big \langle \bfa \cdot \bfn, | \ddiv 
		\bfu |^2 \big \rangle_{\Gamma} + 2 \mu \big \langle \bfa \cdot \bfn, 
		| \ep(\bfu)|^2 \big \rangle_{\Gamma} \\
	& \qquad - (d - 2) \Big( \lambda \| \ddiv \bfu \|^2_{L^2(\Om)} + 2 \mu \| 
		\ep(\bfu) \|^2_{L^2(\Om)} \Big) \\
	& \qquad + d \om^2 \rho \| \bfu \|^2_{L^2(\Om)} - \om^2 \rho \big \langle 
		\bfa \cdot \bfn, | \bfu |^2 \big \rangle_{\Gamma} \\
	& \qquad - 2 \om \im \big \langle A \bfu, \bfv \big \rangle_{\Gamma}.
\end{align*}
Rearranging the terms gives
\begin{align} \label{Eq:GWCc}
	&d \om^2 \rho \| \bfu \|^2_{L^2(\Om)} - (d - 2) \Big( \lambda \| \ddiv \
		\bfu \|^2_{L^2(\Om)} + 2 \mu \| \ep(\bfu) \|^2_{L^2(\Om)} \Big) \\
	& \qquad = -\lambda \big \langle \bfa \cdot \bfn, | \ddiv \bfu |^2 \big 
		\rangle_{\Gamma} - 2 \mu \big \langle \bfa \cdot \bfn, | \ep(\bfu)|^2 
		\big \rangle_{\Gamma} + \om^2 \rho \big \langle \bfa \cdot \bfn, | 
		\bfu |^2 \big \rangle_{\Gamma} \notag \\
	& \qquad \qquad + 2 \om \im \big \langle A \bfu, \bfv \big
		\rangle_{\Gamma} + 2 \re a( \bfu, \bfv). \notag
\end{align}
Adding \eqref{Eq:GWCc} and $(d - 1)$ times \eqref{Eq:GWCa} yields
\begin{align*}
	& \om^2 \rho \| \bfu \|^2_{L^2(\Om)} + \lambda \| \ddiv \bfu \|
		^2_{L^2(\Om)} + 2 \mu \| \ep (\bfu) \|^2_{L^2(\Om)} \\
	& \qquad = -\lambda \big \langle \bfa \cdot \bfn, | \ddiv \bfu |^2 \big 
		\rangle_{\Gamma} - 2 \mu \big \langle \bfa \cdot \bfn, | \ep(\bfu)|^2 
		\big \rangle_{\Gamma} + \om^2 \rho \big \langle \bfa \cdot \bfn, | 
		\bfu |^2 \big \rangle_{\Gamma} \\
	& \qquad \qquad + 2 \om \im \big \langle A \bfu, \bfv \big
		\rangle_{\Gamma} + \re a\big( \bfu, (d - 1) \bfu + 2 \bfv \big). 
\end{align*}
Now applying \eqref{StarShape2}, \eqref{SPDInequality}, and Young's inequality yields
\begin{align*}
	& \om^2 \rho \| \bfu \|^2_{L^2(\Om)} + \lambda \| \ddiv \bfu \|
		^2_{L^2(\Om)} + 2 \mu \| \ep (\bfu) \|^2_{L^2(\Om)} \\
	& \qquad \leq -c_0 \Big( \lambda \| \ddiv \bfu \|^2_{L^2(\Gamma)} + 2 \mu 
		\| \ep(\bfu) \|^2_{L^2(\Gamma)} \Big) + R \om^2 \rho \| \bfu \|
		^2_{L^2(\Gamma)} \\
	& \qquad \qquad + \frac{R^2 \om^2 C_A}{\delta} \| \bfu \|^2_{L^2(\Gamma)} 
		+ \delta \| \bnabla \bfu \|^2_{L^2(\Gamma)} + \re a\big( \bfu, (d - 
		1) \bfu + 2 \bfv \big).
\end{align*}
Noting that $\bfu \in \bV_{\tilde{K}}$ and choosing $\delta = \frac{c_0 \mu}
{\tilde{K}}$ yields
\begin{align*}
	& \om^2 \rho \| \bfu \|^2_{L^2(\Om)} + \lambda \| \ddiv \bfu \|
		^2_{L^2(\Om)} + 2 \mu \| \ep (\bfu) \|^2_{L^2(\Om)} + c_0 \Big( 
		\lambda \| \ddiv \bfu \|^2_{L^2(\Gamma)} + 2 
		\mu \| \ep(\bfu) \|^2_{L^2(\Gamma)} \Big) \\
	& \qquad \leq \frac{1}{\om c_A} \Bigl( R \om^2 \rho + \frac{R^2 \om^2 C_A 
		\tilde{K}}{c_0 \mu} + c_0 \mu \Bigr) \om c_A \| \bfu \|
		^2_{L^2(\Gamma)} + \re a \big( \bfu, (d-1) \bfu + 2 \bfv \big) \\
	& \qquad \leq \Bigl( \frac{R \om^2 \rho c_0 \mu + R^2 \om C_A \tilde{K} + 
		c_0^2 \mu^2}{\om c_A c_0 \mu} \Bigr) \om \big \langle A \bfu, \bfu 
		\big \rangle_{\Gamma} + \re a \big( \bfu, (d-1) \bfu + 2 \bfv \big) 
		\\
	& \qquad = \big| \re a \big( \bfu, \hat{\bfv} \big) \big| + M \big| \im a 
		\big( \bfu, \bfu \big) \big|.
\end{align*}
where $\hat{\bfv} := (d-1) \bfu + 2 \bfv$. Thus
\begin{align} \label{Eq:GWCd}
	\| \bfu \|^2_E \leq \big| \re a \big( \bfu, \hat{\bfv} \big) \big| + M 
		\big| \im a \big( \bfu, \bfu \big) \big|.
\end{align}

By the definitions of  $\| \cdot \|_E$ and $||| \cdot |||_{L^2(\Om)}$ and using 
\eqref{KornInequality} we obtain
\begin{align*}
	||| \hat{\bfv} |||^2_{L^2(\Om)} &= \om^2 \rho \| \hat{\bfv} \|
		^2_{L^2(\Om)} + c_0 \mu \| \hat{\bfv} \|^2_{\Gamma} \\
	& \leq 4 R^2 \om^2 \rho \| \bnabla \bfu \|^2_{L^2(\Om)} + 4 R^2 c_0 \mu 
		\| \bnabla \bfu \|^2_{\Gamma} \\
	& \qquad + (d-1)^2 \om^2 \rho \| \bfu \|_{L^2(\Om)} + (d - 1)^2 c_0 \mu 
		\| \bfu \|^2_{L^2(\Gamma)} \\
	& \leq \left[ 4R^2K \left(1 + \frac{\om^2 \rho}{2 \mu}\right) + 4R^2 
		\tilde{K} + (d-1)^2 \right] \| \bfu \|^2_{E}.
\end{align*}
Thus
\begin{align} \label{Eq:GWCe}
	||| \hat{\bfv} |||_{L^2(\Om)} \leq \gamma \| \bfu \|_E.
\end{align}
It follows from \eqref{Eq:GWCd} and \eqref{Eq:GWCe} that
\begin{align*}
	\sup_{\bfv \in \bV} \frac{| \im a(\bfu, \bfv)|}{\| \bfv \|_E} + 
		\sup_{\bfv \in \bH^1(\Om)} \frac{|\re a(\bfu,\bfv)|}{||| \bfv |||
		_{L^2(\Om)}} 
	&\geq \frac{| \im a(\bfu, \bfu)|}{\| \bfu \|_E} + \frac{|\re 
		a(\bfu,\hat{\bfv})|}{||| \hat{\bfv} |||_{L^2(\Om)}} \\
	&\geq \frac{M | \im a(\bfu, \bfu)| +  |\re a(\bfu,\hat{\bfv})|}
		{\gamma \| \bfu \|_E} \\
	&\geq \frac{1}{\gamma} \| \bfu \|_E.
\end{align*}
The proof is complete.
\end{proof}

The generalized weak coercivity property of $a(\cdot,\cdot)$ given in Theorem 
\ref{thm:GenWeakCoercivity} immediately infers the following stability 
estimate for the solution to \eqref{Eq:WeakForm1a}.

\begin{theorem}
Let $\Om$ be a strictly star-shaped domain with $diam(\Om) = R$. Furthermore, 
suppose there exists some positive constant $\tilde{K}$ such that $\bfu \in 
\bV_{\tilde{K}}$ solves \eqref{Eq:WeakForm1a} for $\bff \in \bL^2(\Om)$ and $
\bfg \in \bL^2(\Gamma)$.  Then the following stability result holds:
\begin{align} \label{SolutionEstimate1}
	\| \bfu \|_E \leq 2 \gamma \left( \frac{1}{\om^2 \rho} + \frac{1}{c_0 
		\mu} \right) \Big( \| \bff \|_{L^2(\Om)} + \| \bfg \|_{L^2(\Gamma)} 
		\Big),
\end{align} 
where $\| \cdot \|_E$ and $\gamma$ are the same as those defined in Theorem 
\ref{thm:GenWeakCoercivity}.
\end{theorem}
\begin{proof}
By the Cauchy-Schwarz inequality, for any $\bfv \in \bH^1(\Om)$ the following 
result holds
\begin{align*}
	| a(\bfu, \bfv)| & = \big| (\bff, \bfv) + \langle \bfg, \bfv 
		\langle_{\Gamma} \big| \\
	& \leq \| \bff \|_{L^2(\Om)} \| \bfv \|_{L^2(\Om)} + \| \bfg \|
		_{L^2(\Gamma)} \| \bfv \|_{L^2(\Gamma)} \\
	& \leq \Big( \| \bff \|^2_{L^2(\Om)} + \| \bfg \|^2_{L^2(\Gamma)} 
		\Big)^{\frac{1}{2}} \Big( \| \bfv \|^2_{L^2(\Om)} + \| \bfv \|
		^2_{L^2(\Gamma)} \Big)^{\frac{1}{2}} \\
	& \leq \Big( \frac{1}{\om^2 \rho} + \frac{1}{c_0 \mu} \Big) ||| \bfv |||
		_{L^2(\Om)}  \Big( \| \bff \|_{L^2(\Om)} + \| \bfg \|
		_{L^2(\Gamma)} \Big) \\
	& \leq \Big( \frac{1}{\om^2 \rho} + \frac{1}{c_0 \mu} \Big) \| \bfv \|
		_{E}  \Big( \| \bff \|_{L^2(\Om)} + \| \bfg \|
		_{L^2(\Gamma)} \Big),
\end{align*}
Combining this inequality with the result of Theorem \ref{thm:GenWeakCoercivity} yields 
\begin{align*}
	\frac{1}{\gamma} \| \bfu \|_E &\leq \sup_{\bfv \in \bV} \frac{| \im 
		a(\bfu, \bfv) |}{\| \bfv \|_E} + \sup_{\bfv \in \bH^1(\Om)} \frac{| 
		\re a(\bfu, \bfv) |}{||| \bfv |||_{L^2(\Om)}} \\
	& \leq 2 \Big( \frac{1}{\om^2 \rho} + \frac{1}{c_0 \mu} \Big) 
		\Big( \| \bff \|_{L^2(\Om)} + \| \bfg \|_{L^2(\Gamma)} 
		\Big).
\end{align*}
The proof is complete.
\end{proof}

\begin{remark}
We remark that both estimates \eqref{GenWeakCoercivity} and \eqref{SolutionEstimate1} still hold 
without Conjecture \ref{conj}, but the dependence of the upper bounds on $\omega$ would be slightly 
stronger as shown in \cite{Cummings_Feng06}.
\end{remark}

%%%%%%%%%%%%%%%%
\section{Formulation of the IP-DG scheme}\label{sec-2}

In this section we shall present our IP-DG scheme for the elastic Helmholtz equations.  
Conceptually, the formulation will mimic those for the scalar Helmholtz 
problem and the Maxwell problem given in \cite{Feng_Wu09,Feng_Wu11}. To do so
we first need to introduce some notation. 

%Let  ${\bf L}^2(Q) = (L^2(Q))^d$ for $Q \subset \Om$ and
%${\bf L}^2(\Sigma) = (L^2(\Sigma))^d$ for $\Sigma \subset \pa Q$ denote the complex-valued inner 
%product spaces with inner products $(\cdot,\cdot)_Q$ and $\langle \cdot, \cdot \rangle_\Sigma$,
%respectively. 
Let $\mathcal{T}_h$ be a shape regular triangulation of the domain $\Om$ such that 
for each triangle/tetrahedron $K \in \mathcal{T}_h$, $h \geq h_K := \mbox{diam}(K)$.  
Also for each edge/face $e$ of triangle $K$ define $h_e := \mbox{diam}(e)$.  Next we 
define the following notation for the set of edges of $\mathcal{T}_h$:
\begin{align}
&\mathcal{E}^I_h:=\mbox{set of all interior edges/faces of } \mathcal{T}_h, \\ 
&\mathcal{E}^B_h:=\mbox{set of all boundary edges/faces of } \mathcal{T}_h \mbox{ on } \Gamma. 
\end{align} 

For each $e \in \mathcal{E}^I_h$ there exists $K_e, K_e' \in \mathcal{T}_h$ 
such that $e = \pa K_e \cap \pa K_e'$ and thus we define the jump and average 
operators on $e$ as follows:
\begin{align*}
[\bfv]|_e &:= \begin{cases}
\bfv|_{K_e} - \bfv|_{K_e'}, 
& \mbox{if the global labeling number of $K_e$ is greater than that of $K_e'$}, \\
\bfv|_{K_e'} - \bfv|_{K_e}, & \mbox{if the global labeling number of $K_e'$ 
is greater than that of $K_e$},
\end{cases} 
\end{align*}
and
\begin{align*}
\{ \bfv \}|_e &:= \frac{1}{2} \big(\bfv|_{K_e} + \bfv|_{K_e'} \big).
\end{align*}
Note that for $e \in \mathcal{E}^B_h$, we use the convention 
$[\bfv]|_e = \{ \bfv \}|_e := \bfv|_e$.  Also for $e \in \mathcal{E}^I_h$ 
we shall need to define the outward normal vector $\bfn_e$ to $e = \pa K_e \cap \pa K_e'$.  
Let $\bfn_e$ be the unit outward normal vector to $K_e$ on $e$ if $K_e$ has 
a greater labeling number than $K_e'$, and the unit outward normal vector to $K_e'$ 
if the opposite is true. 

Next, we introduce the energy space $\bfE$ and the sesquilinear 
form $a_h(\cdot,\cdot)$ on $\bfE \times \bfE$ as follows:
\begin{align}
&\bfE := \prod_{K \in \mathcal{T}_h} {\bf H}^2(K) 
:= \prod_{K \in \mathcal{T}_h} \left(H^2(K)\right)^d, \label{energy}\\
\nonumber \\
&a_h(\bfu, \bfv) := b_h(\bfu, \bfv) + \bfi \Big( J_0(\bfu,\bfv) + J_1(\bfu, \bfv) \Big) 
\qquad \forall \ \bfu, \ \bfv \ \in \bfE, \label{b+form}
\end{align}
where
\begin{align*}
b_h(\bfu, \bfv) &:= \sum_{K \in \mathcal{T}_h} \Bigl( \lambda \big(\ddiv \bfu, \ddiv \bfv \big)_K 
+ 2 \mu \big(\varepsilon(\bfu), \varepsilon(\bfv) \big)_K \Bigr) 
- \sum_{e \in \mathcal{E}^I_h} \bigl\langle 
\{ \sigma(\bfu) \bfn_e \}, [\bfv] \big\rangle_e \\
& \qquad \qquad + \eta \sum_{e \in \mathcal{E}^I_h} 
\big \langle [\bfu],  \{\sigma(\bfv) \bfn_e \}\big \rangle_e, \\
J_0 (\bfu, \bfv) &:= \sum_{e \in \mathcal{E}^I_h} \frac{\gamma_{0,e}}{h_e} 
\big \langle[\bfu], [\bfv] \big \rangle, \\
J_1 (\bfu, \bfv) &:= \sum_{e \in \mathcal{E}^I_h} \gamma_{1,e} h_e \big\langle 
[\sigma(\bfu) \bfn_e], [\sigma(\bfv) \bfn_e] \big\rangle .
\end{align*}
Here $\eta$ is an $h$-independent ``symmetrization" parameter and $\gamma_{0,e}$, 
and $\gamma_{1,e}$ are nonnegative ``penalty" parameters that will be specified later. 
For the rest of this paper we shall restrict ourselves to the case $\eta = -1$ 
(i.e. the symmetric case) but we note that other values of $\eta$ are possible.

\begin{remark}
(a) Both $\bfi J_0$ and $\bfi J_1$ terms are called interior penalty terms. 
To the best our knowledge, penalizing the jumps of the normal stress 
$\sigma(\bfu) \bfn_e$ across the element interfaces seems have not been 
used before, and this term will play an important 
role in our stability and error analysis to be given in the next two sections.

(b) Another new feature of the sesquilinear form $a_h$ is that the penalty 
constants in $\bfi J_0$ and $\bfi J_1$ are complex numbers and must have
positive imaginary parts. This feature will be vital for constructing our 
unconditionally stable IP-DG scheme.

(c) It is easy to see that $a_h(\cdot, \cdot)$ is a consistent discretization 
of the elastic operator $-\ddiv(\sigma(\cdot))$.
\end{remark}

\medskip
We shall also define some appropriate semi-norms/norms on the energy space $\bfE$.
\begin{align*}
| \bfv |_{1,h} &:= \left ( \sum_{K \in \mathcal{T}_h} \lambda \| \ddiv \bfv \|_{L^2(K)}^2 
+ 2\mu \| \varepsilon(\bfv) \|_{L^2(K)}^2 \right)^{ \frac{1}{2}}, \\
\|\bfv \|_{1,h} &:=  \left( |\bfv |^2_{1,h} 
+ \sum_{e \in \mathcal{E}_h^I} \Big( \frac{\gamma_{0,e}}{h_e} \big \| [\bfv ] \big\|^2_{L^2(e)} 
+ \gamma_{1,e} h_e \big \| [\sigma(\bfv) \bfn_e] \big \|^2_{L^2(e)} \Big) \right)^{\frac{1}{2}}, \\
|||\bfv|||_{1,h} &:= \left( \| \bfv \|^2_{1,h} + \sum_{e \in \mathcal{E}_h^I}\frac{h_e}{\gamma_{0,e}}  \big \| \{\sigma(\bfv) \bfn_e\} \big \|^2_{L^2(e)} \right)^{\frac{1}{2}}.
\end{align*}

The weak formulation for \eqref{Eq:ElasticPDE}--\eqref{Eq:ElasticBC} based on
the sesquilinear form $a_h(\cdot,\cdot)$ is then defined as seeking 
$\bfu \in \bfE \cap {\bf H}^2_{\mbox{\scriptsize loc}}(\Om)\cap \bH^1(\Ome)$ such that
for all $\bfv \in \bfE \cap { \bf H }^2_{\mbox{\scriptsize loc}}(\Om)\cap \bH^1(\Ome)$
\begin{align}
a_h(\bfu, \bfv) - \om^2 \rho \big(\bfu, \bfv)_\Om 
+ \bfi \om \big \langle A \bfu, \bfv \big \rangle_\Gamma 
= \big ( \bff ,\bfv \big)_\Om + \big \langle \bfg, \bfv \big \rangle_\Gamma.
%\qquad \forall \bfv \in \bfE \cap { \bf H }^d_{\mbox{\scriptsize loc}}(\Om) 
\label{Eq:WeakPDE}
\end{align}

Our discontinuous Galerkin finite element space is taken as the standard piecewise 
linear polynomial space 
\begin{align*}
\bfV_h := \prod_{K \in \mathcal{T}_h} \bfP_1(K) 
:= \prod_{K \in \mathcal{T}_h} \big( P_1(K) \big)^d ,
\end{align*}
where $P_1(K)$ is the space of linear polynomials with complex coefficients 
on the triangle/tetrahedron $K \in \mathcal{T}_h$. Then our IP-DG method is
defined as seeking $\bfu_h \in \bfV_h$ such that
\begin{align}
A_h(\bfu_h, \bfv_h) = \big ( \bff ,\bfv_h \big)_\Om + \big \langle \bfg, \bfv_h \big\rangle_\Gamma 
\qquad \forall \bfv_h \in \bfV_h \label{IP-DG Method},
\end{align}
where
\begin{align}
A_h(\bfu_h, \bfv_h) := a_h(\bfu_h, \bfv_h) - \om^2 \rho \big(\bfu_h, \bfv_h)_\Om 
+ \bfi \om \big \langle A \bfu_h, \bfv_h \big \rangle_\Gamma. \label{Sesquilinear Form}
\end{align}

In the following sections we shall establish stability and error estimates 
for the IP-DG method defined above.  To prove these estimates we shall require 
that there exist $h$, $\gamma_0$, and $\gamma_1 > 0$ such that
\begin{align*}
	 h_K \leq h \leq Ch_K& \qquad \forall K \in \mathcal{T}_h, \\
	\gamma_{0,e} \leq \gamma_0 \leq C \gamma_{0,e}& \qquad \forall e \in \mathcal{E}_h^I, \\
	\gamma_{1,e} \leq \gamma_1 \leq C \gamma_{1,e}& \qquad \forall e \in \mathcal{E}_h^I, 
\end{align*}
for some positive constant $C$ independent of $h$.  It will also be convenient to introduce
the notation $\xi := 1 + \gamma_0^{-1}$.  It is easy to show the following 
norm equivalence on $\bfV_h$:
%%%
\begin{align}
\| \bfv_h \|_{1,h} \leq ||| \bfv_h |||_{1,h} \leq C \xi^{\frac{1}{2}} \| \bfv_h \|_{1,h} 
\qquad \forall \bfv_h \in \bfV_h. \label{Norm Equivalence}
\end{align}

%%%%%%%%%%%%%%%
\section{Stability estimates}\label{sec-3}
The goal of this section is to establish unconditional stability estimates for the 
solution of our IP-DG method, which can be regarded as the discrete counterpart
of the following frequency-explicit estimate for the solution 
$\bfu \in \bH^2(\Om)$ of the problem \eqref{Eq:ElasticPDE}--\eqref{Eq:ElasticBC},
which was established in \cite{Cummings_Feng06}.
\begin{theorem} \label{thm:PDE_estimate}
Suppose that $\Om$ is a convex polygonal domain or a smooth domain and
$\bfu \in \bH^2(\Om)$ solves \eqref{Eq:ElasticPDE}--\eqref{Eq:ElasticBC}. 
Then $\bfu$ satisfies the following estimate:
\begin{align}\label{PDE_estimate}
\| \bfu \|_{H^2(\Om)} \leq C \Big( \om^\alpha + \frac{1}{\om^2} \Big) 
\big( \| \bff \|_{L^2(\Om)} + \| \bfg \|_{L^2(\Gamma)} \big),
\end{align}
where $\alpha = 1$ if there exists $\tilde{K} > 0$ such that $\bfu \in 
\bV_{\tilde{K}}$ as defined in Section \ref{sec-1b}, and $\alpha = 2$ 
otherwise.
\end{theorem}

\begin{remark}
The desired Korn-type inequality on the boundary $\Gamma$ of a general domain 
$\Ome$ required for $\bfu \in \bV_{\tilde{K}}$
is yet be proved and is posed as a conjecture in \cite{Cummings_Feng06}, although there 
are strong evidences to indicate that this is the case. Moreover, we believe that the optimal dependence
on $\omega$ (i.e., when $\alpha = 1$) should hold in \eqref{PDE_estimate} as that is 
is the case for the scalar Helmholtz equation and the time-Harmonic Maxwell's equations.  
For the rest of this paper we shall keep this parameter $\alpha$.

In \cite{Lorton_14}, an a priori error estimate is obtained in the case 
$\om^{1+\alpha} h = O(1)$ (called the asymptotic mesh regime), with stability being a consequence.  
This analysis is based on the so-called Schatz argument \cite{Schatz74}.  Similar results 
in the asymptotic mesh regime have been proved for the standard finite element method 
applied to the scalar and elastic Helmholtz equations 
(c.f. \cite{Cummings01,Aziz_Kellogg79,Douglas_Santos_Sheen_Bennethum93,Douglas_Sheen_Santos94}).
In this paper we are more interested in deriving solution estimates in the pre-asymptotic mesh 
regime, i.e. $\om^\beta h =O(1)$ for some $1\leq \beta<1+\alpha$.  This is the most interesting case
because similar results have not been proved for the standard finite element method for
the elastic Helmholtz problem and the analysis requires a non-standard argument.
\end{remark}
 
We begin by first proving some discrete weak coercivity properties 
for the sesquilinear form $A_h(\cdot, \cdot)$.  The desired stability estimates 
can then be obtained as a consequence of these discrete weak coercivity properties.  
%We note that, since the matrix $A$ is a real symmetric positive definite matrix, 
%there exist positive constants $c_A$ and $C_A$ such that
%%
%\begin{align}
%c_A \| \bfv_h \|^2_{L^2(\Gamma)} \leq \big \langle A \bfv_h, \bfv_h \big \rangle_\Gamma 
%\leq C_A \| \bfv_h \|^2_{L^2(\Gamma)} \qquad \forall \bfv_h \in \bfV_h. \label{SPD Inequality}
%\end{align}

Let us start with a technical lemma which will be instrumental in establishing 
the discrete coercivity of the sesquilinear form $A_h(\cdot,\cdot)$.  We note that 
this lemma makes use of the fact that the space $\bfV_h$ is made up of piecewise linear 
polynomials, and thus a different strategy must be sought in the case that $\bfV_h$ 
is a higher order DG finite element space.
\begin{lemma} \label{Stability Lemma 1}
For any $0 < \delta < 1$ and $\bfv_h \in \bfV_h$ there holds
\begin{align} \label{Stability Lemma 1 Inequality}
|\bfv_h |^2_{1,h} & \leq \delta \om^2 \rho \| \bfv_h \|^2_{L^2(\Om)} 
+ \frac{C_\delta}{\om h} \om \| \bfv_h \|^2_{L^2(\Gamma)} 
+ \frac{C_\delta}{\gamma_0} J_0(\bfv_h,\bfv_h) \\
&\qquad
+ \frac{C_\delta}{\om^2 \rho h^2 \gamma_1} J_1(\bfv_h, \bfv_h), \notag
\end{align}
where $C$ is a positive constant independent of $\om$, $h$, $\gamma_0$, and $\gamma_1$.
\end{lemma}

\begin{proof}
We note that $\bfv_h|_K \in \bfP_1(K)$ for all $K \in \mathcal{T}_h$ and thus 
$\bdiv \sigma(\bfv_h) |_K = {\bf 0}$.  For any $\bfw_h, \bfv_h \in \bV_h$ 
and $K \in \mathcal{T}_h$, integrating by parts yields
\begin{align*}
0 &= \big (\bdiv \sigma(\bfv_h), \bfw_h \big)_K \\
& = - \big \langle \sigma(\bfv_h) \bfn_K, \bfw_h \big \rangle_{\pa K} 
+ \lambda \big( \ddiv \bfv_h, \ddiv \bfw_h \big)_K 
+ 2 \mu \big( \varepsilon(\bfv_h), \varepsilon(\bfw_h) \big)_K.
\end{align*}
Here we have used a differential identity from Lemma 3 of \cite{Cummings_Feng06}, namely 
\begin{align}
\sigma(\bfv_h):\nabla \overline{\bfw_h} = \lambda \ddiv \bfv_h \ddiv \overline{\bfw_h} 
+ 2\mu \varepsilon(\bfv_h):\varepsilon(\overline{\bfw_h}). \label{Cummings Identity 1}
\end{align}
Thus
\begin{align*}
\lambda \big( \ddiv \bfv_h, \ddiv \bfw_h \big)_K 
+ 2 \mu \big( \varepsilon(\bfv_h), \varepsilon(\bfw_h) \big)_K 
= \big \langle \sigma(\bfv_h) \bfn_K, \bfw_h \big \rangle_{\pa K}.
\end{align*}
Setting $\bfw_h = \bfv_h$ and summing over all $K \in \mathcal{T}_h$ gives
\begin{align*}
|\bfv_h |^2_{1,h}  &= \sum_{K \in \mathcal{T}_h} \big \langle \sigma(\bfv_h) \bfn_K, 
\bfv_h \big \rangle_{\pa K} \\
& = \sum_{e \in \mathcal{E}_h^I} \Big( \big \langle \{ \sigma(\bfv_h) \bfn_e \}, 
[ \bfv_h ] \big \rangle_e + \big \langle [ \sigma(\bfv_h) \bfn_e ], 
\{ \bfv_h \} \big \rangle_e \Big) + \sum_{e \in \mathcal{E}_h^B} 
\big \langle \sigma(\bfv_h) \bfn_e, \bfv_h \big \rangle_{e} \\
& \leq \sum_{e \in \mathcal{E}^I_h} \left( \| \{ \sigma(\bfv_h) \bfn_e \} 
\|_{L^2(e)} \| [\bfv_h] \|_{L^2(e)} + \| [ \sigma(\bfv_h) \bfn_e ] \|_{L^2(e)} 
\| \{\bfv_h\} \|_{L^2(e)} \right) \\
& \qquad + \sum_{e \in \mathcal{E}_h^B} \|\sigma(\bfv_h) \bfn_e \|_{L^2(e)} \| \bfv_h \|_{L^2(e)}.
\end{align*}
Applying the trace and inverse inequalities on the right-hand side yields
\begin{align*}
| \bfv_h |^2_{1,h} & \leq C \sum_{e \in \mathcal{E}_h^B} h^{-\frac{1}{2}}_e 
\| \sigma(\bfv_h) \|_{L^2(K_e)} \| \bfv_h \|_{L^2(e)} \\
& \qquad + C \sum_{e \in \mathcal{E}^I_h} h^{-\frac{1}{2}}_e \| [\bfv_h] 
\|_{L^2(e)} \left( \| \sigma(\bfv_h) \|_{L^2(K_e)} + \| \sigma(\bfv_h) \|_{L^2(K_e')} \right) \\
& \qquad + C \sum_{e \in \mathcal{E}^I_h} h^{-\frac{1}{2}}_e 
\| [\sigma(\bfv_h) \bfn_e] \|_{L^2(e)} \left( \| \bfv_h \|_{L^2(K_e)} 
+ \| \bfv_h \|_{L^2(K_e')} \right).
\end{align*}

Finally, it follows from the discrete Cauchy-Schwarz inequality along with Young's inequality that
\begin{align*}
| \bfv_h |^2_{1,h}  & \leq C h^{-  \frac{1}{2}} \left( \sum_{K \in \mathcal{T}_h} 
\| \sigma(\bfv_h) \|_{L^2(K)} \right)^{\frac{1}{2}} \| \bfv_h \|_{L^2(\pa \Om)} \\
& \qquad + C \gamma_0^{-\frac{1}{2}} \Big(  \sum_{e \in \mathcal{E}_h^I} \frac{\gamma_{0,e}}{h_e}  
\|  [ \bfv_h] \|^2_{L^2(e)} \Big)^{  \frac{1}{2}} \Big( \sum_{K \in \mathcal{T}_h} 
\| \sigma( \bfv_h) \|^2_{L^2(K)} \Big)^{\frac{1}{2}} \\ 
& \qquad + C \gamma_1^{-  \frac{1}{2}} h^{-  1} \Big(  \sum_{e \in \mathcal{E}_h^I} \gamma_{1,e}{h_e}
\|[\sigma(\bfv_h) \bfn_e ] \|^2_{L^2(e)} \Big)^{ \frac{1}{2}} \| \bfv_h \|_{L^2(\Om)} \\
& \leq \delta | \bfv_h |^2_{1,h} + \frac{C(\lambda + 2 \mu)}{\delta \om h} \om 
\| \bfv_h \|^2_{L^2(\pa \Om)} + \frac{\delta}{2} | \bfv_h |^2_{1,h} 
+ \frac{C(\lambda + 2 \mu)}{ \delta \gamma_0}J_0(\bfv_h, \bfv_h) \\
& \qquad \qquad + \delta (1 - \delta) \om^2 \rho \| \bfv_h \|^2_{L^2(\Om)} 
+ \frac{C}{\delta (1 - \delta) \om^2 \rho h^2 \gamma_1} J_1(\bfv_h, \bfv_h).
\end{align*}
Thus \eqref{Stability Lemma 1 Inequality} holds.
\end{proof}

We are now ready to state the following main theorem of this section.
 
\begin{theorem} \label{Discrete Coercivity Theorem}
Let $A_h(\cdot, \cdot)$ be the sesquilinear form defined in \eqref{Sesquilinear Form}.  
Then there hold for all $\bfv_h \in \bfV_h$ 
\begin{align}
| A_h(\bfv_h, \bfv_h) | &\geq C \Big(\xi + \frac{1}{\om h} 
+ \frac{1}{\om^2 h^2 \gamma_1} \Big)^{-1} \Big(| \bfv_h|^2_{1,h} 
+ \om^2 \rho \| \bfv_h \|^2_{L^2(\Om)} \Big), \label{Discrete Coercivity Inequality a} \\
| A_h(\bfv_h, \bfv_h) | &\geq J_0(\bfv_h ,\bfv_h) + J_1(\bfv_h, \bfv_h) 
+ \om \big \langle A \bfv_h, \bfv_h \big \rangle_\Gamma, \label{Discrete Coercivity Inequality b}
\end{align}
where $\xi = 1 + \gamma_0^{-1}$ and $C$ is a positive constant independent 
of $\om$, $h$, $\gamma_0$, and $\gamma_1$.
\end{theorem}

\begin{proof}
Let $\bfv_h\in \bfV_h$. Taking the real and imaginary parts of $A_h(\bfv_h, \bfv_h)$ yields
\begin{align}
\re A_h(\bfv_h, \bfv_h) &= | \bfv_h |^2_{1,h} - \om^2 \rho \| \bfv_h \|^2_{L^2(\Om)} 
- 2 \re \sum_{e \in \mathcal{E}_h^I} \big \langle \{ \sigma(\bfv_h) \bfn_e \}, 
[\bfv_h] \big \rangle_e, \label{Real Part of A_h} \\
\im A_h(\bfv_h, \bfv_h) &=  J_0(\bfv_h ,\bfv_h) + J_1(\bfv_h, \bfv_h) 
+ \om \big \langle A \bfv_h, \bfv_h \big \rangle_\Gamma. \label{Imaginary Part of A_h}
\end{align}
Then \eqref{Discrete Coercivity Inequality b} follows directly from \eqref{Imaginary Part of A_h}.

To verify \eqref{Discrete Coercivity Inequality a}, we need to bound the term 
$\sum_{e \in \mathcal{E}_h^I} \big \langle \{ \sigma(\bfv_h) \bfn_e \}, [\bfv_h] \big \rangle_e$,
which involves using the trace and inverse inequality and was already carried out  
previously in the proof of Lemma \ref{Stability Lemma 1}. Thus
\begin{align*}
\re A_h(\bfv_h, \bfv_h) & \leq | \bfv_h |^2_{1,h} - \om^2 \rho \| \bfv_h \|^2_{L^2(\Om)} 
+ C \gamma_0^{-\frac{1}{2}} \Big( \sum_{e \in \mathcal{E}^I_h} \frac{\gamma_0}{h_e} 
\| [ \bfv_h ] \|^2_{L^2(\Om)} \Big)^{\frac{1}{2}} | \bfv_h |_{1,h} \\
& \leq \frac{3}{2} | \bfv_h |^2_{1,h} - \om^2 \rho \| \bfv_h \|^2_{L^2(\Om)} 
+ \frac{C}{\gamma_0} J_0 (\bfv_h, \bfv_h).
\end{align*}
Combining the above inequality with \eqref{Stability Lemma 1 Inequality} 
and using $\delta = \frac{1}{4}$ we get
\begin{align*}
&\frac{1}{2} | \bfv_h |^2_{1,h} + \om^2 \rho \| \bfv_h \|^2_{L^2(\Om)} 
\leq - \re A_h(\bfv_h, \bfv_h) + 2 | \bfv_h|^2_{1,h} + \frac{C}{\gamma_0} J_0(\bfv_h, \bfv_h) \\
& \qquad \qquad \leq - \re A_h(\bfv_h, \bfv_h) + \frac{1}{2} \om^2 \rho \| \bfv_h \|^2_{L^2(\Om)} 
+ \frac{C}{\om h c_A} \om c_A \| \bfv_h \|^2_{L^2(\pa \Om)} \\
& \qquad \qquad \qquad + \frac{C}{\gamma_0} J_0(\bfv_h, \bfv_h) 
+ \frac{C}{\om^2 \rho h^2 \gamma_1} J_1(\bfv_h, \bfv_h).
\end{align*}
Thus, subtracting both sides of the above inequality by 
$\frac{1}{2} \om^2 \rho \| \bfv_h \|^2_{L^2(\Om)}$ and using both 
\eqref{SPDInequality} and \eqref{Imaginary Part of A_h} yield
\begin{align*}
&| \bfv_h |^2_{1,h} + \om^2 \rho \| \bfv_h \|^2_{L^2(\Om)} \\
& \qquad \leq - 2 \re A_h(\bfv_h, \bfv_h) + C \Big( \frac{1}{\gamma_0} 
+ \frac{1}{\om h c_A} + \frac{1}{\om^2 \rho h^2 \gamma_1} \Big) \im A_h(\bfv_h, \bfv_h) \\
&\qquad \leq C \Big( 1+ \frac{1}{\gamma_0}+ \frac{1}{\om h c_A} 
+ \frac{1}{\om^2 \rho h^2 \gamma_1} \Big) | A_h(\bfv_h, \bfv_h)|.
\end{align*}
Hence, \eqref{Discrete Coercivity Inequality a} holds. The proof is complete.
\end{proof}

\begin{remark}
\eqref{Discrete Coercivity Inequality a}--\eqref{Discrete Coercivity Inequality b} 
is called weak coercivity because of the complex magnitude (instead of the real part)
is used in the left-hand sides of these inequalities.
\end{remark}

%\begin{theorem}
%For any $\om,h,\gamma_0, \gamma_1 > 0$, $\bff \in \bL^2(\Om)$, and 
%$\bfg \in \bL^2(\pa \Om)$ there exists a unique solution $\bfu_h$ of \eqref{IP-DG Method}.
%\end{theorem}
%
%\begin{proof}
%It is easy to show that the sesquilinear form $A_h(\cdot,\cdot)$ is continuous 
%on $\bV_h \times \bV_h$. Then the above result is a consequence of the continuity 
%of $A_h(\cdot,\cdot)$, Theorem \ref{Discrete Coercivity Theorem}, and the well-known 
%Lax-Milgram-Babu\v{s}ka theorem \cite{Babuska_Aziz_72}.
%\end{proof}

As a consequence of the discrete weak coercivity of Theorem \ref{Discrete Coercivity Theorem} 
we readily obtain the following stability estimates.

\begin{theorem} \label{Stability Theorem}
Suppose that $\bfu_h \in \bfV_h$ solves the IP-DG method given by \eqref{IP-DG Method}.  
Then the following inequalities hold:
\begin{align}
&|\bfu_h |^2_{1,h} + \om^2 \rho \| \bfu_h \|^2_{L^2(\Om)} 
\leq C \Big( \csta^2 \| \bff \|^2_{L^2(\Om)} 
+ \csta \| \bfg \|^2_{L^2(\Gamma)} \Big) \label{Stability Theorem Inequality a} \\
&J_0(\bfu_h, \bfu_h) + J_1(\bfu_h, \bfu_h) + \big \langle A \bfu_h, \bfu_h \big \rangle_\Gamma 
\leq \frac{C}{\om} \Big( \csta \| \bff \|^2_{L^2(\Om)} 
+\| \bfg \|^2_{L^2(\Gamma)} \Big), \label{Stability Theorem Inequality b}
\end{align}
where 
\begin{equation}\label{Csta}
\csta = \frac{\xi}{\om} + \frac{1}{\om^2 h} + \frac{1}{\om^3 h^2 \gamma_1}, 
\qquad\qquad
\xi = 1 + \frac{1}{\gamma_0},
\end{equation}
and $C$ is a positive constant independent of $\om$, $h$, $\gamma_0$, and $\gamma_1$.
\end{theorem}

\begin{proof}
By \eqref{Discrete Coercivity Inequality a} and \eqref{Discrete Coercivity Inequality b} we get
\begin{align*}
&| \bfu_h |^2_{1,h} + \om^2 \rho \| \bfu_h \|^2_{L^2(\Om)} + \om \Big( 1 
+ \frac{1}{\gamma_0}+ \frac{1}{\om h} 
+ \frac{1}{\om^2 \rho h^2 \gamma_1} \Big) \big \langle A \bfu_h ,\bfu_h \big \rangle_{\Gamma} \\
& \qquad \leq C  \Big( 1 + \frac{1}{\gamma_0} + \frac{1}{\om h} 
+ \frac{1}{\om^2 \rho h^2 \gamma_1} \Big) | A_h(\bfu_h, \bfu_h) | \\
& \qquad \leq C  \Big( 1 + \frac{1}{\gamma_0}+ \frac{1}{\om h} 
+ \frac{1}{\om^2 \rho h^2 \gamma_1} \Big) \Big( \| \bff \|_{L^2(\Om)} \| \bfu_h \|_{L^2(\Om)} 
+ \| \bfg \|_{L^2(\Gamma)} \| \bfu_h \|_{L^2(\Gamma)} \Big) \\
& \qquad \leq \frac{C}{\om^2 \rho} \Big(1 + \frac{1}{\gamma_0}+ \frac{1}{\om h} 
+ \frac{1}{\om^2 \rho h^2 \gamma_1} \Big)^2 \| \bff \|^2_{L^2(\Om)} 
+ \frac{\om^2 \rho}{2} \| \bfu_h \|^2_{L^2(\Om)} \\
& \qquad \qquad \qquad + \frac{C}{\om c_A} \Big( 1 + \frac{1}{\gamma_0}
+ \frac{1}{\om h} + \frac{1}{\om^2 \rho h^2 \gamma_1} \Big) \| \bfg \|^2_{L^2(\Gamma)} \\
& \qquad \qquad \qquad + \om c_A \Big( 1 + \frac{1}{\gamma_0}+ \frac{1}{\om h} 
+ \frac{1}{\om^2 \rho h^2 \gamma_1} \Big) \| \bfu_h \|^2_{L^2(\Gamma)}.
\end{align*}
Substituting \eqref{SPDInequality} into the above inequality infers 
\eqref{Stability Theorem Inequality a}.

Now, combining \eqref{Discrete Coercivity Inequality b} with 
\eqref{Stability Theorem Inequality a} yields
\begin{align*}
& J_0(\bfu_h, \bfu_h) + J_1(\bfu_h, \bfu_h)+\om \big \langle A \bfu_h, \bfu_h \big \rangle_{\Gamma} \\
& \qquad \leq | A_h(\bfu_h, \bfu_h) | \\
& \qquad \leq \| \bff \|_{L^2(\Om)} \| \bfu_h \|_{L^2(\Om)} 
+ \| \bfg \|_{L^2(\Gamma)} \| \bfu_h \|_{L^2(\Gamma)} \\
& \qquad \leq \frac{C}{\om \rho^{\frac{1}{2}}}\| \bff \|_{L^2(\Om)} 
\bigg( \frac{1}{ \rho} \csta^2 \| \bff \|^2_{L^2(\Om)} 
+ \frac{1}{ c_A}  \csta \| \bfg \|^2_{L^2(\Gamma)} \bigg)^{\frac{1}{2}} \\ 
& \qquad \qquad + \frac{C}{\om c_A} \| \bfg \|^2_{L^2(\Gamma)} 
+ \frac{\om c_A}{2} \| \bfu_h \|^2_{L^2(\Gamma)} \\
& \qquad \leq \frac{C}{\om \rho} \csta \| \bff \|^2_{L^2(\Om)} 
+ \frac{C}{\om c_A} \| \bfg \|^2_{L^2(\Gamma)} + \frac{\om c_A}{2} \| \bfu_h \|^2_{L^2(\Gamma)}.
\end{align*}
Using \eqref{SPDInequality} in the above inequality infers the 
desired estimate \eqref{Stability Theorem Inequality b}.
\end{proof}

\begin{remark}\label{rem-Csta}
When $h$ is restricted to the pre-asymptotic mesh regime characterized by 
the constraint $\om^\beta h =O(1)$ for $1\leq \beta<1+\alpha$, then we have 
\begin{align}
\csta \leq  C \Big( \frac{\xi}{\om} + \om^{\beta - 2} + \frac{\om^{2 \beta - 3}}{\gamma_1} \Big).
\end{align}
Thus, the constant in the above stability estimate can be replaced by one that 
is independent of $h$ (but still depends on $\Ome$) in the pre-asymptotic mesh regime.
\end{remark}

We conclude this section by establishing the following existence and uniqueness theorem.

\begin{theorem}\label{existence}
For any $\om,h,\gamma_0, \gamma_1 > 0$, $\bff \in \bL^2(\Om)$, and
$\bfg \in \bL^2(\pa \Om)$ there exists a unique solution $\bfu_h$ of \eqref{IP-DG Method}.
\end{theorem}

\begin{proof}
Due to the linearity of the PDE, it is easy to check that \eqref{IP-DG Method} is 
equivalent to a linear (algebraic) system. The stability estimate \eqref{Stability Theorem Inequality a}
immediately implies that the linear system only has the trivial solution when
$\mathbf{f}=\mathbf{0}$ and $\mathbf{g}=\mathbf{0}$, which then infers that
the coefficient matrix of the linear system is nonsingular. Thus, the 
system must have a unique solution and consequently \eqref{IP-DG Method}
has a unique solution.
\end{proof}

%%%%%%%%%%%%%%%%%%%%
\section{Error estimates}\label{sec-4}
The goal of this section is to provide error estimates that hold in all mesh regimes,
in particular, in the pre-asymptotic mesh regime.  In order to accomplish this we shall 
adapt a non-standard argument of \cite{Feng_Wu09,Feng_Wu14} by defining an elliptic 
projection operator onto $\bfV_h$ and establishing estimates for the error between the 
solution $\bfu \in \bfE \cap \bfH^d(\Om)$ of \eqref{Eq:ElasticPDE}--\eqref{Eq:ElasticBC} 
and its elliptic projection, and utilizing the unconditional stability estimate to control 
the error between the projection and the IP-DG solution. In this section we assume 
$\bfu\in \bfH^2(\Ome)$ satisfying the estimate given in Theorem \ref{thm:PDE_estimate}.

First, we show that the sesquilinear form $a_h(\cdot, \cdot)$ is continuous 
and satisfies a weak coercivity condition.
\begin{theorem} \label{Continuity and Coercivity}
For any $\bfv, \bfw \in \bfE$ there exists a positive constant $C$ independent 
of $\om$, $h$, $\gamma_0$, $\gamma_1$ such that
\begin{align}
|a_h(\bfv, \bfw)| &\leq C ||| \bfv |||_{1,h} ||| \bfw |||_{1,h}. \label{Continuity a}
\end{align}
Also for any $0 < \delta < 1$ and $\bfv_h \in \bfV_h$ there holds
\begin{align}
\re a_h(\bfv_h, \bfv_h) + \left(1 - \delta + \frac{C_\delta}{\gamma_0} \right) 
\im a_h(\bfv_h,\bfv_h) \geq (1 - \delta) \| \bfv_h \|^2_{1,h} \label{Weak Coercivity}. 
\end{align}
\end{theorem}

\begin{proof}
\eqref{Continuity a} is easy to prove, we omit it and leave it to the interested reader.  
To show \eqref{Weak Coercivity}, we note that
\begin{align*}
\re a_h(\bfv_h, \bfv_h) &\geq |\bfv_h|^2_{1,h} + 2 \sum_{e \in \mathcal{E}^I_h}\Big| 
\big \langle [\bfv_h], \{\sigma(\bfv_h) \bfn_e \} \big \rangle_e \Big| \\
&\geq |\bfv_h|^2_{1,h} - 2 \sum_{e \in \mathcal{E}^I_h}  \big \| [\bfv_h] 
\big \|_{L^2(e)} \big \| \{\sigma(\bfv_h) \bfn_e \} \big \|_{L^2(e)} \\
&\geq |\bfv_h|^2_{1,h} - \sum_{e \in \mathcal{E}^I_h} \frac{C}{\gamma_0} 
\frac{\gamma_0}{h_e^{\frac{1}{2}}} \big \| [\bfv_h] \big \|_{L^2(e)} 
\Big( \big \|\sigma(\bfv_h) \big \|_{L^2(K_e)} + \big \|\sigma(\bfv_h) \big \|_{L^2(K_e')} \Big) \\
&\geq (1-\delta)|\bfv_h|^2_{1,h} -\frac{C_\delta}{\gamma_0}J_0(\bfv_h,\bfv_h).
\end{align*}
Also,
\begin{align*}
\im a_h(\bfv_h,\bfv_h) = J_0(\bfv_h, \bfv_h) + J_1(\bfv_h,\bfv_h).
\end{align*}
\eqref{Weak Coercivity} follows from combining the above two results.
\end{proof}

%%%%%%%%%%%%%%%%
\subsection{Elliptic projection and its error estimates} \label{Projection Subsection}
For any $\bfw \in E$ we define its elliptic projection $\tilde{\bfw}_h \in \bfV_h$ as 
the solution to the following problem:
\begin{align}
a_h(\tilde{\bfw}_h,\bfv_h) + \bfi \om \big \langle A \tilde{\bfw}_h,\bfv_h \big \rangle_\Gamma 
= a_h(\bfw,\bfv_h) + \bfi \om \big \langle A \bfw, \bfv_h \big \rangle_\Gamma 
\qquad \forall \bfv_h \in \bfV_h. \label{Elliptic Projection}
\end{align}

\begin{remark}
By Theorem \ref{Continuity and Coercivity}, the sesquilinear form 
$a_h(\cdot,\cdot) + \i \om \langle A \cdot, \cdot \rangle_\Gamma$ is both continuous 
and coercive when $\gamma_0$ is taken to be sufficiently large.  Thus, the Lax-Milgram 
Theorem ensures that $\tilde{\bfw}_h$ is well-defined for such a choice of $\gamma_0$.
\end{remark}

By the consistency of $a_h(\cdot, \cdot)$, it is trivial to check that
the following Galerkin orthogonality property holds.
\begin{lemma} \label{Galerkin Orthogonality}
Suppose that $\bfw \in \bfE$ and $\tilde{\bfw}_h \in \bfV_h$ is its elliptic projection, then
\begin{align}
a_h(\bfw - \tilde{\bfw}_h,\bfv_h) + \bfi \om \big \langle A(\bfw - \tilde{\bfw}_h), 
\bfv_h \big \rangle = 0 \qquad \forall \bfv_h \in \bfV_h. \label{Galerkin Orthogonality Equation}
\end{align}
\end{lemma}

Now we are ready to estimate the error between the solution $\bfu$ of 
\eqref{Eq:ElasticPDE}--\eqref{Eq:ElasticBC} and its elliptic projection $\bftu_h$.

\begin{theorem} \label{Elliptic Projection Error}
Let $\bfu \in \bH^2(\Om)$ solve \eqref{Eq:ElasticPDE}--\eqref{Eq:ElasticBC} and 
let $\bftu_h \in \bfV_h$ be its elliptic projection defined in \eqref{Elliptic Projection}.  
Then there hold the following estimates:
\begin{align} \label{Elliptic Projection Error Inequality a}
&||| \bfu - \bftu_h |||_{1,h} + \om^{\frac{1}{2}} \xi \| \bfu - \bftu_h \|_{L^2(\Gamma)} \\
& \qquad \leq C \xi^2 h \big (\xi + \gamma_1 + \om h \big)^{\frac{1}{2}} \left(\om^\alpha 
+ \frac{1}{\om^2} \right) \Big( \| \bff \|_{L^2(\Om)} + \| \bfg \|_{L^2(\Gamma)} \Big). \notag\\
&\| \bfu - \bftu_h \|_{L^2(\Om)}
\leq C \xi^2 h^2 \big (\xi + \gamma_1 + \om h \big) 
\left( \om^\alpha + \frac{1}{\om^2} \right) \Big( \| \bff \|_{L^2(\Om)} 
+ \| \bfg \|_{L^2(\Gamma)} \Big), \label{Elliptic Projection Error Inequality b}
\end{align}
where $\xi = 1 + \gamma_0^{-1}$ and $C$ is a positive constant independent of 
$\om$, $h$, $\gamma_0$, and $\gamma_1$.
\end{theorem}

\begin{proof}
Let $\bfhu_h \in \bf\bfV_h$ denote the $P_1$-conforming finite element interpolant 
of $\bfu$ on $\mathcal{T}_h$. Then the following estimates are well-known 
(c.f. \cite{Brenner_Scott,Ciarlet}):
\begin{align} \label{Interpolation Error 1}
\| \bfu - \bfhu_h \|_{L^2(\Om)} &\leq Ch^2 | \bfu |_{H^2(\Om)},\\ 
\|\bnabla ( \bfu - \bfhu_h) \|_{L^2(\Om)} &\leq Ch | \bfu |_{H^2(\Om)}.\label{Interpolation Error 1b}
\end{align}
By the trace and inverse inequalities we get 
\begin{align}
&|||\bfu - \bfhu_h |||_{1,h} \leq C(\xi + \gamma_1)^{\frac{1}{2}}h|\bfu|_{H^2(\Om)},
\label{Interpolation Error 2} \\
&\|\bfu - \bfhu_h \|_{L^2(\Gamma)} \leq Ch^{\frac{3}{2}} |\bfu |_{H^2(\Om)}. 
\label{Interpolation Error 3}
\end{align}
Set $\bfpsi_h := \bftu_h - \bfhu_h$.  By Galerkin orthogonality along with the fact that 
$\bfpsi_h + \bfu - \bftu_h = \bfu - \bfhu_h$ we obtain
\begin{align*}
a_h(\bfpsi_h, \bfpsi_h) + \bfi \om \big\langle A \bfpsi_h, \bfpsi_h \big\rangle_{\Gamma} 
= a_h(\bfu - \bfhu_h, \bfpsi_h) +\bfi\om \big \langle A(\bfu -\bfhu_h),\bfpsi_h\big\rangle_{\Gamma}.
\end{align*}

Next, it follows from \eqref{Norm Equivalence} and \eqref{Weak Coercivity} 
with $\delta = \frac{1}{2}$ (noting that $C_{\frac{1}{2}} > \frac{1}{2}$) that
\begin{align*}
&\frac{1}{2} ||| \bfpsi_h |||^2_{1,h}  \leq C\xi \| \bfpsi_h \|^2_{1,h} \\
& \qquad \leq C \xi \re a_h(\bfpsi_h, \bfpsi_h) + C \xi \left(\frac{1}{2} 
+ \frac{C_{\frac{1}{2}}}{\gamma_0} \right) \im a_h(\bfpsi_h, \bfpsi_h) \\
& \qquad = C\xi \re \Big( a_h(\bfpsi_h, \bfpsi_h) 
+ \bfi \om \big \langle A\bfpsi_h, \bfpsi_h \big \rangle_\Gamma \Big) \\
& \qquad \qquad + C\xi \left(\frac{1}{2} 
+ \frac{C_{\frac{1}{2}}}{\gamma_0} \right) \im \Big( a_h(\bfpsi_h, \bfpsi_h) 
+ \bfi \om \big \langle A\bfpsi_h, \bfpsi_h \big \rangle_\Gamma \Big) \\
& \qquad \qquad - C\xi \om \left(\frac{1}{2} 
+ \frac{C_{\frac{1}{2}}}{\gamma_0} \right) \big \langle A\bfpsi_h, \bfpsi_h \big \rangle_\Gamma \\
& \qquad = C\xi \re \Big( a_h(\bfu - \bfhu_h, \bfpsi_h) 
+ \bfi \om \big \langle A(\bfu - \bfhu_h), \bfpsi_h \big \rangle_\Gamma \Big)  \\
& \qquad \qquad + C\xi \left(\frac{1}{2} + \frac{C_{\frac{1}{2}}}{\gamma_0} \right) 
\im \Big( a_h(\bfu - \bfhu_h, \bfpsi_h) + \bfi \om \big \langle A(\bfu - \bfhu_h), 
\bfpsi_h \big \rangle_\Gamma \Big) \\
& \qquad \qquad - C\xi \om \left(\frac{1}{2} 
+ \frac{C_{\frac{1}{2}}}{\gamma_0} \right) \big \langle A\bfpsi_h, \bfpsi_h \big \rangle_\Gamma \\
& \qquad \leq C\xi  ||| \bfpsi_h |||_{1,h} ||| \bfu - \bfhu |||_{1,h} 
+ C \om \xi \| \bfpsi_h \|_{L^2(\Gamma)} \| \bfu - \bfhu \|_{L^2(\Gamma)} 
- C \om \xi^2 \| \bfpsi_h \|^2_{L^2(\Gamma)} \\
& \qquad \qquad + C \xi^2 \Big( ||| \bfpsi_h |||_{1,h} ||| \bfu - \bfhu |||_{1,h} 
+ \om \| \bfpsi_h \|_{L^2(\Gamma)} \| \bfu - \bfhu \|_{L^2(\Gamma)} \Big) \\
&\qquad \leq C \xi^2 ||| \bfpsi_h |||_{1,h} ||| \bfu - \bfhu |||_{1,h} 
+ 2C \om \xi^2 \| \bfu - \bfhu_h \|^2_{L^2(\Gamma)} 
- \frac{C}{4} \om \xi^2 \| \bfpsi_h \|_{L^2(\Gamma)} \\
&\qquad \leq C \xi^4 |||\bfu - \bfhu_h |||^2_{1,h} 
+ \frac{1}{4} ||| \bfpsi_h |||^2_{1,h} - \frac{C}{4}\om \xi^2 \| \bfpsi_h \|^2_{L^2(\Gamma)} 
+2C \om \xi^2 \| \bfu - \bfhu_h \|^2_{L^2(\Gamma)}.
\end{align*}
Substituting \eqref{Interpolation Error 2} and \eqref{Interpolation Error 3} into the 
above estimate gives
\begin{align*}
&||| \bfpsi_h |||^2_{1,h} + \om \xi^2 \| \bfpsi_h \|^2_{L^2(\Gamma)} 
\leq C \Big( \xi^4 |||\bfu -\bfhu_h|||^2_{1,h} +\om \xi^2 \| \bfu -\bfhu_h \|^2_{L^2(\Gamma)}\Big) \\
& \qquad \qquad \leq C \Big (\xi^4 (\xi + \gamma_1) h^2 | \bfu |^2_{H^2(\Om)} 
+ \om \xi^2h^3 |\bfu|^2_{H^2(\Om)} \Big) \\
& \qquad \qquad = C \xi^4 h^2 \big( \xi + \gamma_1 + \om h\big) | \bfu |^2_{H^2(\Om)} \\
& \qquad \qquad \leq C \xi^4 h^2 \big(\xi + \gamma_1 + \om h \big) \left( \om^\alpha 
+ \frac{1}{\om^2} \right)^2 \Big( \| \bff \|^2_{L^2(\Om)} + \| \bfg \|^2_{L^2(\Gamma)} \Big).
\end{align*}
Thus
\begin{align*}
&||| \bfpsi_h |||_{1,h} + \om^{\frac{1}{2}} \xi \| \bfpsi_h \|_{L^2(\Gamma)}  \\
& \qquad \qquad \leq C \xi^2 h \big(\xi + \gamma_1 
+ \om h \big)^{\frac{1}{2}} \left( \om^\alpha 
+ \frac{1}{\om^2} \right) \Big( \| \bff \|_{L^2(\Om)} + \| \bfg \|^2_{L^2(\Gamma)} \Big).
\end{align*}

Recall that $\bfu - \bftu_h = \bfu - \bfhu_h - \bfpsi_h$.  By the triangle inequality we get
\begin{align*}
&||| \bfu - \bftu_h |||_{1,h} + \om^{\frac{1}{2}} \xi \| \bfu - \bftu_h \|_{L^2(\Gamma)} \\
& \qquad \qquad\leq ||| \bfu - \bfhu_h |||_{1,h} 
+ \om^{\frac{1}{2}} \xi \| \bfu - \bfhu_h \|_{L^2(\Gamma)}  
+ ||| \bfpsi_h |||_{1,h} + \om^{\frac{1}{2}} \xi \| \bfpsi_h \|_{L^2(\Gamma)} \\
& \qquad \qquad \leq C \xi^2h \big(\xi + \gamma_1 + \om h \big)^{\frac{1}{2}} 
\left( \om^\alpha + \frac{1}{\om^2} \right) 
\Big( \| \bff \|_{L^2(\Om)} + \| \bfg \|^2_{L^2(\Gamma)} \Big).
\end{align*}
Therefore, \eqref{Elliptic Projection Error Inequality a} holds.

To obtain \eqref{Elliptic Projection Error Inequality b}, we appeal to the duality 
argument by considering the following auxiliary PDE problem:  
\begin{align*}
-\bdiv ( \sigma(\bfw)) = \bfu - \bftu_h&  \qquad \mbox{in } \Om, \\
\sigma(\bfw) \bfn - \bfi \om \bfw = {\bf 0}& \qquad \mbox{on } \Gamma.
\end{align*}
It can be shown that there exists a unique solution $\bfw \in \bH^2(\Om)$ such that
\begin{align}
\|\bfw \|_{H^2(\Om)}  \leq C \| \bfu - \bftu_h \|_{L^2(\Om)}
\end{align}
Let $\bfhw_h \in \bfV_h$ denote the $P_1$-conforming finite element interpolant 
of $\bfw$ over $\mathcal{T}_h$.  Then by testing the above auxiliary problem with 
$\bfu - \bftu_h$ we get
\begin{align*}
&\| \bfu - \bftu_h \|^2_{L^2(\Om)} = - \big( \bfu - \bftu_h, \bdiv \sigma(\bfw) \big)_\Om \\
& \qquad= a_h(\bfu - \bftu_h,\bfw) 
+ \bfi \om \big \langle A(\bfu - \bftu_h), \bfw \big \rangle_\Gamma \\
& \qquad = a_h(\bfu - \bftu_h, \bfw - \bfhw_h) 
 \bfi \om \big \langle A(\bfu - \bftu_h), \bfw - \bfhw_h \big \rangle_\Gamma \\
& \qquad \leq C \Big( ||| \bfu - \bftu_h |||_{1,h}||| \bfw - \bfhw_h|||_{1,h} 
+ \om \| \bfu - \bftu_h \|_{L^2(\Gamma)} \| \bfw - \bfhw_h \|_{L^2(\Om)} \Big) \\
& \qquad \leq C \Big( (\xi + \gamma_1)^{\frac{1}{2}} h \| \bfw \|_{H^2(\Om)} 
||| \bfu - \bftu_h |||_{1,h} + \om h^{\frac{3}{2}} \| \bfw \|_{H^2(\Om)} 
\| \bfu - \bftu_h \|_{L^2(\Gamma)} \Big) \\
& \qquad \leq C \Big( (\xi + \gamma_1)^{\frac{1}{2}} h \| \bfu - \bftu_h \|_{L^2(\Om)} 
||| \bfu - \bftu_h |||_{1,h} 
+ \om h^{\frac{3}{2}} \| \bfu - \bftu_h \|_{L^2(\Om)} \| \bfu - \bftu_h \|_{L^2(\Gamma)} \Big).
\end{align*}
Hence,
\begin{align*}
\| \bfu - \bftu_h \|_{L^2(\Om)} 
&\leq C h \bigl( \xi + \gamma_1 + \om h\bigr)^{\frac{1}{2}} \Big( ||| \bfu - \bftu_h |||_{1,h} 
+ \om^{\frac{1}{2}} \xi \| \bfu - \bftu_h \|_{L^2(\Om)} \Big) \\
& \leq C \xi^2 h^2 \bigl(\xi + \gamma_1 + \om h \bigr) \left(\om^\alpha + \frac{1}{\om^2} \right) 
\big(\| \bff \|_{L^2(\om)}+ \| \bfg \|^2_{L^2(\Gamma)} \big).
\end{align*}
Thus, \eqref{Elliptic Projection Error Inequality b} holds.
\end{proof}

%%%%%%%%%%%%%%%%%%%%%%%%%%%%
\subsection{Error estimates for the IP-DG method}
The goal of this subsection is to obtain upper bounds for the error $\bfe_h := \bfu - \bfu_h$, 
where $\bfu \in \bH^2(\Om)$ solve \eqref{Eq:ElasticPDE}--\eqref{Eq:ElasticBC} and 
$\bfu_h$ is its IP-DG approximation defined by \eqref{IP-DG Method}.  We shall accomplish 
this goal by using the stability estimate from Theorem \ref{Stability Theorem} and the 
elliptic projection studied in Subsection \ref{Projection Subsection}.  
Subtracting \eqref{IP-DG Method} from \eqref{Eq:WeakPDE} immediately infers the following 
Galerkin orthogonality for $\bfe_h$:
\begin{align}
a_h(\bfe_h, \bfv_h) - \om^2 \rho \big(\bfe_h, \bfv_h \big)_\Om 
+ \bfi \om \big \langle A\bfe_h, \bfv_h \big \rangle_\Gamma = {\bf 0} 
\qquad \forall \bfv_h \in \bfV_h. \label{Error Identity 1}
\end{align}
Let $\bftu_h \in \bfV_h$ be the elliptic projection of $\bfu$ as defined 
in \eqref{Elliptic Projection}.  Next we define $\bfpsi := \bfu - \bftu_h$ 
and $\bfphi_h := \bftu_h- \bfu_h$.  Thus, $\bfe_h$ can be decomposed as 
$\bfe_h = \bfpsi + \bfphi_h$. By Galerkin orthogonality given in 
\eqref{Galerkin Orthogonality Equation} and \eqref{Error Identity 1}, 
we have the following identity:
\begin{align} \label{Error Identity 1a} 
a_h(\bfphi_h, \bfv_h) &- \om^2 \rho \big(\bfphi_h, \bfv_h \big)_\Om 
+ \bfi \om \big \langle A \bfphi_h , \bfv_h \big \rangle_\Gamma \\
&= -a_h(\bfpsi, \bfv_h) + \om^2 \rho \big(\bfpsi, \bfv_h \big)_\Om 
- \bfi \om \big \langle A \bfpsi , \bfv_h \big \rangle_\Gamma \notag \\
&= \om^2 \rho \big(\bfpsi, \bfv_h \big)_\Om \qquad \forall \bfv_h \in \bfV_h.\notag
\end{align}
In other words we obtain that $\bfphi_h \in  \bfV_h$ solves \eqref{IP-DG Method} 
with $\bff = \om^2 \rho \bfpsi$ and $\bfg \equiv {\bf 0}$. This allows us to bound
$\bfphi_h$ by using the estimates from Theorem \ref{Stability Theorem} and 
Theorem \ref{Elliptic Projection Error}.  Specifically, we have the next lemma.

\begin{lemma} \label{IP-DG Error Lemma}
Let $\bfu \in \bH^2(\Om)$ solve \eqref{Eq:ElasticPDE}--\eqref{Eq:ElasticBC}, 
$\bfu_h$ be its IP-DG approximation, and $\bftu_h$ its elliptic projection.  
Then $\bfphi_h := \bftu_h - \bfu_h$ satisfies
\begin{align} \label{IP-DG Error Lemma Inequality}
&\| \bfphi_h \|_{1,h} + \om \rho^{\frac{1}{2}} \| \bfphi_h \|_{L^2(\Om)}  \\
& \qquad \qquad 
\leq C \xi^2 \om^2 h^2 \csta \Big( \xi + \gamma_1 + \om h \Big) 
\Big( \om^\alpha +\frac{1}{\om^2} \Big) \Big( \| \bff \|_{L^2(\Om)} +\| \bfg \|_{L^2(\Gamma)} \Big). 
\notag
\end{align}
\end{lemma}

\begin{proof}
As mentioned above, \eqref{Error Identity 1a} implies that $\bfphi_h$ solves 
\eqref{IP-DG Method} with $\bff = \om^2 \rho \bfpsi$ and $\bfg \equiv {\bf 0}$. 
Thus, an application of Theorem \ref{Stability Theorem} yields
\begin{align*}
\|\bfphi_h \|_{1,h} + \om \rho^{\frac{1}{2}} \| \bfphi_h \|_{L^2(\Om)} 
&\leq \frac{C}{ \rho^{\frac{1}{2}}}  \csta  \| \om^2 \rho \bfpsi \|_{L^2(\Om)}\\
&\leq C \om^2 \rho^{\frac{1}{2}} \csta \| \bfpsi \|_{L^2(\Om)}.
\end{align*}
An application of Theorem \ref{Elliptic Projection Error} immediately infers
\eqref{IP-DG Error Lemma Inequality}.
\end{proof}

We are now ready to derive error estimates for our IP-DG method.  The next theorem is a 
consequence of combining Theorem \ref{Elliptic Projection Error} and 
Lemma \ref{IP-DG Error Lemma}.  

\begin{theorem} \label{IP-DG Error Theorem}
Let $\bfu \in \bH^2(\Om)$ solve \eqref{Eq:ElasticPDE}--\eqref{Eq:ElasticBC} and 
$\bfu_h$ be its IP-DG approximation.  Then $\bfu - \bfu_h$ satisfies the 
following estimates:
\begin{align}\label{IP-DG Error Theorem Inequality}
&\| \bfu - \bfu_h \|_{1,h} + \om \rho^{\frac{1}{2}} \| \bfu - \bfu_h \|_{L^2(\Om)}  \\
& \qquad \leq C \xi^2 h \big( \xi + \gamma_1 + \om h \big) 
\Big(\om^\alpha +\frac{1}{\om^2} \Big)\Big(\|\bff\|_{L^2(\Om)}+\|\bfg \|_{L^2(\Gamma)} \Big)\notag \\
& \qquad \qquad + C \xi^2 \om h^2 \big(1 + \om \csta \big) \big( \xi + \gamma_1 + \om h \big) 
\Big( \om^\alpha + \frac{1}{\om^2} \Big) \| \bff \|_{L^2(\Om)}, \notag \\
&\| \bfu - \bfu_h \|_{L^2(\Om)} \label{IP-DG Error Theorem Inequality 2} \\ \notag 
&\qquad \leq C \xi^2 h^2 \big(1 + \om \csta \big)\big( \xi + \gamma_1 + \om h \big)
\Big( \om^{\alpha} +\frac{1}{\om^2} \Big)\Big(\|\bff \|_{L^2(\Om)} + \| \bfg \|_{L^2(\pa \Om)} \Big),
\end{align}
where $\csta$ and $\xi$ are defined in \eqref{Csta}, and $C$ is a positive constant independent of 
$\om, h, \gamma_0, \gamma_1$.
\end{theorem}

\begin{proof}
Recall that $\bfe_h= \bfu - \bfu_h= \bfpsi+\bfphi_h$ and 
the desired estimates for $\bfpsi$ and $\bfphi_h$ have already been established in 
Theorem \ref{Elliptic Projection Error} and Lemma \ref{IP-DG Error Lemma}, respectively. These 
estimates are combined in the following steps to obtain \eqref{IP-DG Error Theorem Inequality}:
\begin{align*}
&\| \bfe_h \|_{1,h} + \om \rho^{\frac{1}{2}} \| \bfe_h \|_{L^2(\Om)} \\
& \qquad \leq |||\bfpsi |||_{1,h} + \om \rho^{\frac{1}{2}} \| \bfpsi \|_{L^2(\Om)} 
+ \| \bfphi_h \|_{1,h} + \om \rho^{\frac{1}{2}} \| \bfphi_h \|_{L^2(\Om)} \\
& \qquad \leq C \xi^2 h \big(\xi + \gamma_1 + \om h \big)^{\frac{1}{2}} 
\Big( \om^\alpha + \frac{1}{\om^2} \Big)^{\frac{1}{2}} 
\Big(\| \bff \|_{L^2(\Om)} + \| \bfg \|_{L^2(\Gamma)} \Big) \\
& \qquad \qquad + C \xi^2 \om h^2 \big(\xi + \gamma_1 + \om h \big)
\Big( \om^\alpha +\frac{1}{\om^2} \Big) \Big(\|\bff \|_{L^2(\Om)} + \| \bfg \|_{L^2(\Gamma)} \Big) \\
& \qquad \qquad + C \xi^2 \om^2 h^2  \csta \big( \xi + \gamma_1 + \om h \big) 
\Big( \om^\alpha + \frac{1}{\om^2} \Big) \Big(\|\bff \|_{L^2(\Om)} + \|\bfg \|_{L^2(\Gamma)} \Big) \\
& \qquad \leq C \xi^2  h  \big( \xi + \gamma_1 + \om h \big) 
\Big( \om^\alpha + \frac{1}{\om^2} \Big) \Big(\|\bff \|_{L^2(\Om)} +\|\bfg \|_{L^2(\Gamma)} \Big) \\
& \qquad \qquad + C \xi^2 \om h^2 \big(1 + \om \csta \big) \big( \xi + \gamma_1 + \om h \big) 
\Big( \om^\alpha + \frac{1}{\om^2} \Big)\Big(\| \bff \|_{L^2(\Om)} + \| \bfg \|_{L^2(\Gamma)} \Big).
\end{align*}
Similarly, \eqref{IP-DG Error Theorem Inequality 2} is obtained by combining 
Theorem \ref{Elliptic Projection Error} and Lemma \ref{IP-DG Error Lemma}.
\end{proof}

\begin{remark} \label{Error Remark}
(a) The above estimates are sub-optimal in general due to the dependence of $\csta$ 
on the mesh size $h$. However, when $h$ belongs to the pre-asymptotic 
mesh regime (i.e. when $\om^\beta h =O(1)$ for $1\leq \beta < 1+\alpha$), 
$\csta$ can be bounded by a constant which is independent of $h$ as explained in
Remark \ref{rem-Csta}. Therefore, the above error estimates become optimal (in $h$)
in the pre-asymptotic mesh regime.  

(b) Estimates that are optimal in $h$ in the asymptotic mesh regime can be found in \cite{Lorton_14}.
\end{remark}

%%%%%%%%%%%%%%%%%%%%%%%%%%%%%%%%%%%%%%%%%%%%%%%%%%%%%%%%%
\section{Numerical experiments} \label{sec:elastic_IPDG_numerics}
In this section, we present some numerical tests to demonstrate key features of the 
proposed IP-DG method. In all our tests we choose $\Om = (-0.5,0.5)^2 \subset \mathbb{R}^2$ 
(i.e. the unit square in $\mathbb{R}^2$ centered at the origin), along with the material 
constants $\rho = \mu = \lambda = 1$, and penalty constants $\gamma_0 = 10$ and $\gamma_1 = 0.1$.  
For the sake of testing the exact error, $\bff$ and $\bfg$ are chosen so that the exact 
solution to the elastic Helmholtz problem is 
$\bfu = \frac{1}{\om^2 r}[e^{\bfi \om r} - 1, e^{-\bfi \om r} - 1]^T$, where $r = \| \bfx \|$
denotes the Euclidean length of $\bfx$.  
This simple problem along with the subsequent numerical tests are chosen to mirror those 
in \cite{Feng_Wu09} for the scalar Helmholtz problem.  
Some sample plots are given in Figures \ref{fig:Ex1} and \ref{fig:Ex2}.  These plots 
demonstrate how well the proposed IP-DG method can capture the wave with large 
frequency when using a relatively coarse mesh. 

To partition the domain $\Om$, a uniform triangulation $\mathcal{T}_h$ is used.  For a 
positive integer $n$, define $\mathcal{T}_{1/n}$ to be a triangulation of $2n^2$ congruent 
isosceles triangles with side lengths $1/n,1/n,$ and $\sqrt{2}/n$.  
Figure \ref{fig:MeshExample} shows the triangulation $\mathcal{T}_{1/10}$.

The numerical tests in this section intend to demonstrate the following:
\begin{itemize}
\item absolute stability of the proposed IP-DG method,
\item error of the proposed IP-DG solution,
\item pollution effect on the error when $\om h = O(1)$,
\item absence of the pollution effect when $\om^3 h^2 = O(1)$,
\item performance comparison between standard FE and the proposed IP-DG method on the test problem.
\end{itemize}

\begin{figure}[htb]
\centering
\includegraphics[scale=0.25]{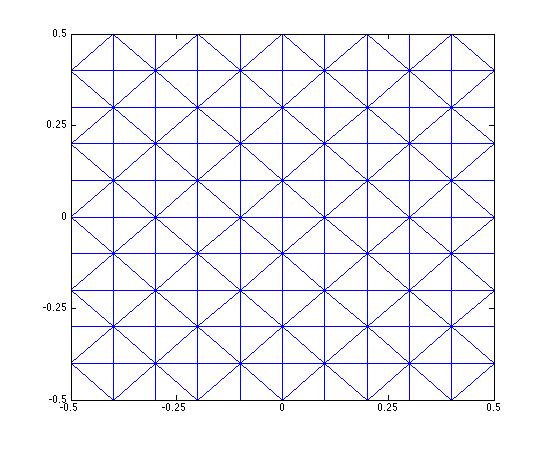}
\caption{The triangulation $\mathcal{T}_{1/10}$. \label{fig:MeshExample}}
\end{figure}

\begin{figure}[htb]
\centerline{\includegraphics[scale = 0.25]{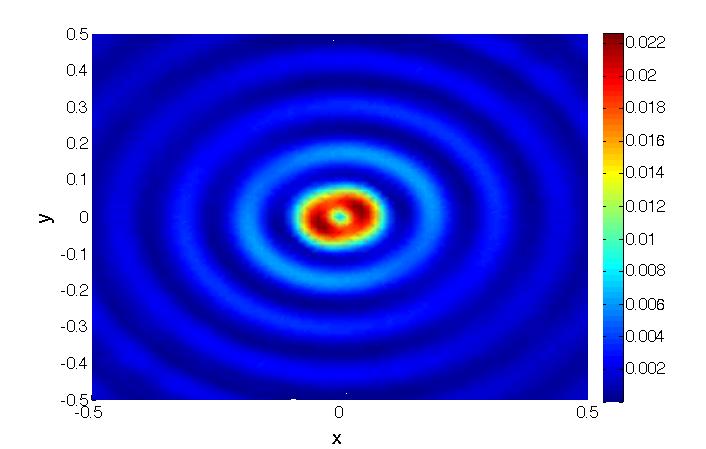} \includegraphics[scale = 0.25]{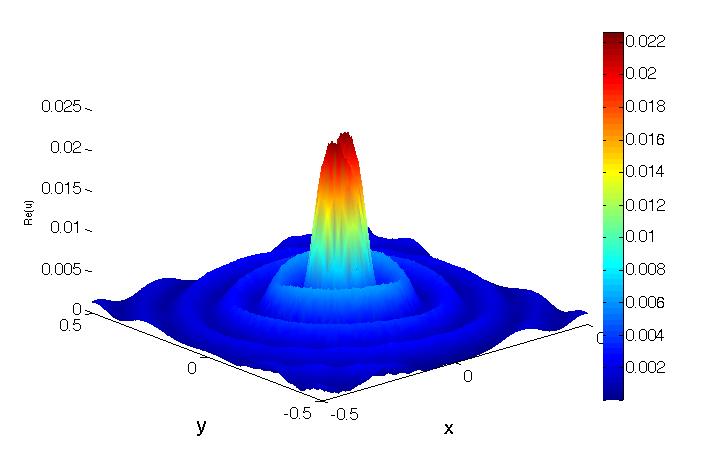}}
\caption{Plot of $\| \re \big(\bfu_h \big) \|_{L^2(\Ome)}$ for $\om = 50$ and $h = 1/70$. 
Both a top down view (left) and a side view (right) are shown. \label{fig:Ex1}}
\end{figure}

\begin{figure}[htb]
\centerline{\includegraphics[scale = 0.25]{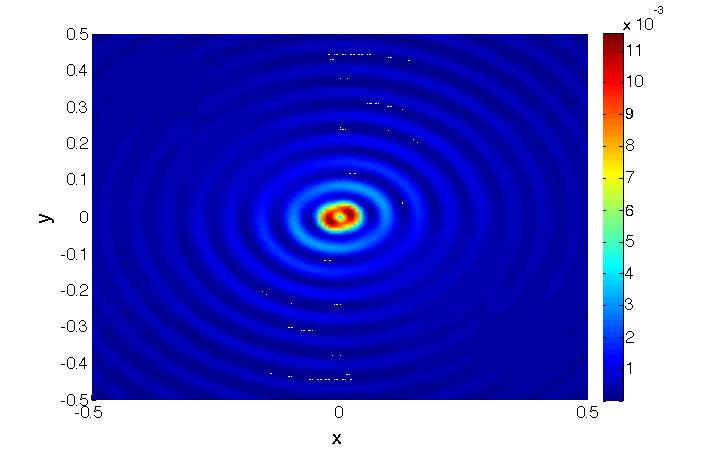} \includegraphics[scale = 0.25]{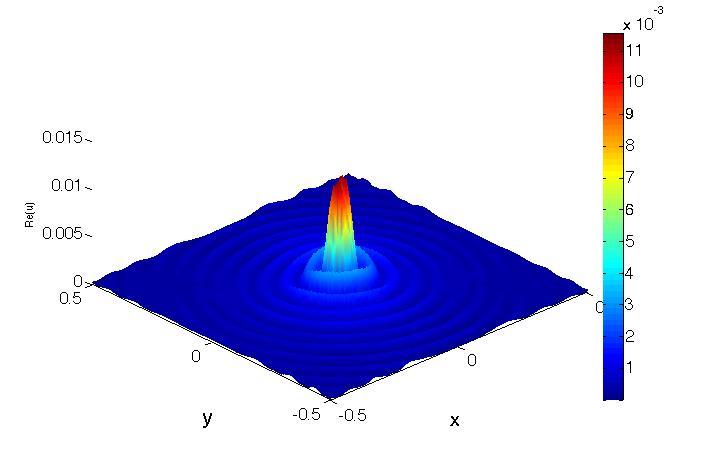}}
\caption{Plot of $\| \re \big(\bfu_h \big) \|_{L^2(\Ome)}$ for $\om = 100$ and $h = 1/120$. 
Both a top down view (left) and a side view (right) are shown. \label{fig:Ex2}}
\end{figure}

%%%%%%%%%%%%%%%%%%%%
\subsection{Stability} \label{subsec:elastic_IPDG_numerical_stability}
In this subsection, the stability of both the proposed IP-DG method and the $P_1$-conforming 
finite element method will be discussed. Let $\bfu^{FEM}_h$ denote the $P_1$-conforming 
finite element approximation of $\bfu$. Recall that the proposed IP-DG approximation 
is unconditionally stable, i.e. it is stable for all $\om, h, \gamma_0, \gamma_1 > 0$.  
This has yet been established for the $P_1$-conforming finite element approximation.  
In fact, the stability of the $P_1$-conforming finite element approximation is only known 
to hold when $h$ satisfies $\om^2 h \leq C$. 

Figure \ref{fig:Stability} plots both $\| \bfu_h \|_{1,h}$ and $\| \bfu^{FEM} \|_{1,h}$ 
for $h = 0.05, 0.01$ and $\om = 1,2,\cdots,$ $200$.  We observe that $\| \bfu_h \|_{1,h}$ decreases 
in a smooth fashion as $\om$ increases.  This smooth behavior of $\| \bfu_h \|_{1,h}$ 
is indicative of the absolute stability of the IP-DG approximation.  On the other hand, 
we observe oscillations in $\| \bfu_h^{FEM} \|_{1,h}$ that occur when we vary $\om$.  
This oscillation is indicative of the instability of the $P_1$-conforming finite 
element method when $h$ is too large.

\begin{figure}[htb]
\centerline{\includegraphics[scale = 0.25]{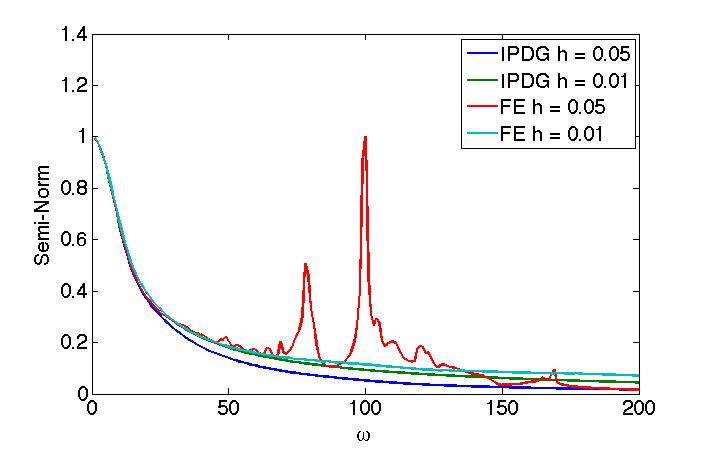} \includegraphics[scale = 0.25]{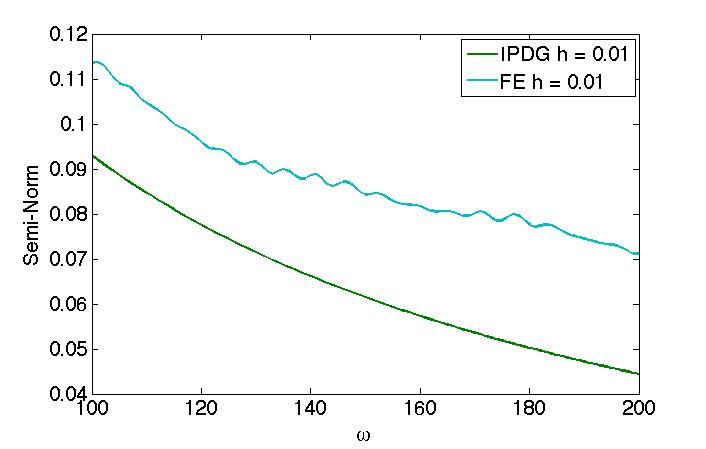}}
\caption{Plots of $\| \bfu_h \|_{1,h}$ and $\| \bfu_h^{FEM} \|_{1,h}$. \label{fig:Stability}}
\end{figure}

%%%%%%%%%%%%%
\subsection{Error} \label{subsec:elastic_IPDG_numerical_error}
In this subsection, the optimal order of convergence for the proposed IP-DG method will 
be demonstrated.  The pollution effect will also be examined.  From 
Theorems \ref{IP-DG Error Theorem} and the error analysis for the asymptotic mesh regime carried 
out in \cite{Lorton_14} (c.f. Remark \ref{Error Remark}) we expect the error 
in $\| \cdot \|_{1,h}$ to decrease at an optimal order in both the pre-asymptotic 
and asymptotic mesh regimes.  In other words, $\| \bfu - \bfu_h \|_{1,h} = O(h)$ 
is expected. Figure \ref{ErrorPlot} is a log-log plot of the relative error 
$\| \bfu - \bfu_h \|_{1,h} / \| \bfu \|_{1,h}$ against the value $1/h$ for 
frequencies $\om = 5, 10, 20, 30$.  From this plot, it is observed that the relative 
error decreases at the same rate as $h$, thus displaying the optimal order of convergence 
in the relative semi-norm.  Also displayed in Figure \ref{ErrorPlot} is the error 
when $\om$ varies according to the constraint $\om h = 0.25$. From this figure it is 
observed that the error increases as $\om$ increases under this constraint.  
This is due to the pollution effect on the error for the elastic Helmholtz problem. 

\begin{figure}[htb]
\centering
\includegraphics[scale=0.25]{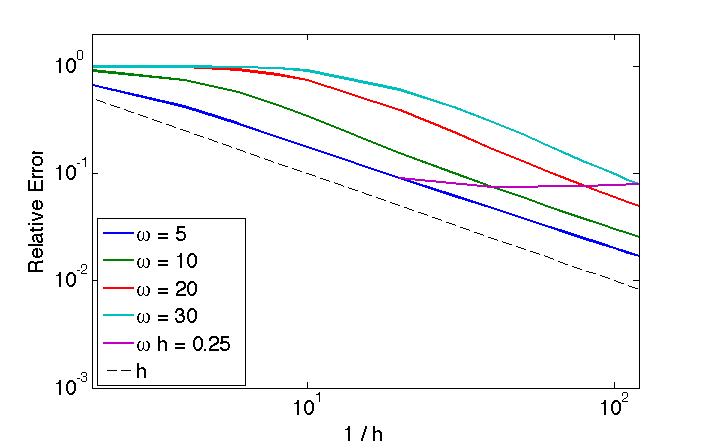}
\caption{Log-log plot of the relative error for the IP-DG approximation measured in 
the broken $H^1$-seminorm for different values of $\omega$. \label{ErrorPlot}}
\end{figure}

The pollution effect for Helmholtz-type problems is characterized by the increase in error 
as $\om$ is increased under the constraint $\om h = O(1)$.  This effect is intrinsic to 
Helmholtz-type problems (c.f. \cite{Ihlenburg_Babuska95}).  It is well-known that the pollution 
effect can be eliminated if $h$ is chosen to fulfill the stronger constraint $\om^3 h^2 = O(1)$.  
In Figure \ref{fig:PolPlots} the relative error is plotted against $\om$ as $h$ is chosen 
to satisfy different constraints.  Under the constraints $\om h = 1$ and $\om h = 0.5$, 
the pollution effect is present and the relative error increases as $\om$ is increased. 
On the other hand, when $\om^3 h^2 = 1$ is used to choose the the mesh size $h$, 
the pollution effect is eliminated.  

\begin{figure}[htb]
\centerline{\includegraphics[scale=0.25]{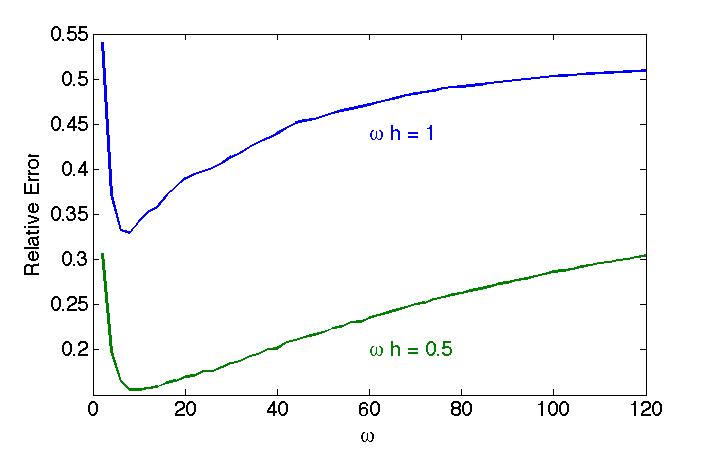} \includegraphics[scale=0.25]{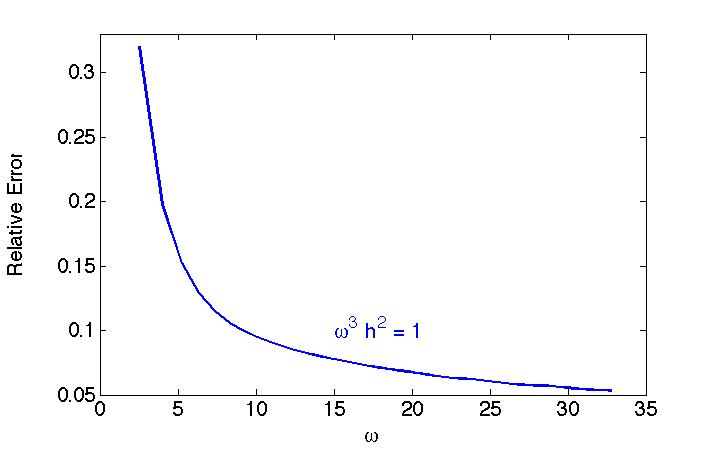}}
\caption{Relative error of the IP-DG approximations measured in the broken $H^1$ seminorm 
computed for different pairs of $\om$ and $h$ combinations.  \label{fig:PolPlots}}
\end{figure}

%%%%%%%%%%%%%%
\subsection{Comparison between the IP-DG method and the FE method} \label{subsec:elastic_IPDG_vs_FEM}
In this subsection, we compare the performance of the proposed IP-DG method and the $P_1$-conforming 
finite element method.  As stated previously, the proposed IP-DG method is unconditionally stable 
while the $P_1$-conforming finite element method is only shown to be stable when $h$ 
satisfies $\om^2 h = O(1)$.  With this in mind, one can anticipate that in the case that the 
frequency $\om$ is large, the IP-DG method should perform better.  

In Figures \ref{fig:IP-DGVSFEM1}--\ref{fig:IP-DGVSFEM3}, $\| \re (\bfu_h) \|_{L^2(\Ome)}$ 
and $\| \re (\bfu^{FEM}_h) \|_{L^2(\Ome)}$ are plotted for $\om = 100$ and $h = 1/50, 1/120, 1/200$ 
on a cross-section over the line $y = x$.  In addition, $\| \re (\bfu) \|_{L^2(\Ome)}$ is plotted 
to measure how well the respective approximations capture the true solution.  In 
Figure \ref{fig:IP-DGVSFEM1}, it is observed that $\bfu_h$ already captures the phase of $\bfu$ 
with $h = 1/50$ while not fully capturing the large changes in magnitude.  On the other hand, 
for $h = 1/50$, $\bfu^{FEM}_h$ has spurious oscillations. In this case, $\bfu^{FEM}_h$ also 
fails to capture the changes in the magnitude of the wave.  In Figure \ref{fig:IP-DGVSFEM2}, 
we see that for $h = 1/120$, $\bfu_h$ captures the phase and changes in magnitude of the wave 
very well while $\bfu^{FEM}_h$ still displays spurious oscillations.  In Figure \ref{fig:IP-DGVSFEM3}, 
we see for $h = 1/200$, both methods capture the wave well. However, the IP-DG method captures 
the wave slightly better.  These examples demonstrate that the IP-DG method approximates 
high frequency waves better than the standard finite element method does, in particular, on a coarse 
mesh. This is of great importance when memory is limited or one wishes to employ a multi-level 
solver such as multigrid or multi-level Schwarz space/domain decomposition methods.

\begin{figure}[htb]
\centerline{\includegraphics[scale = .25]{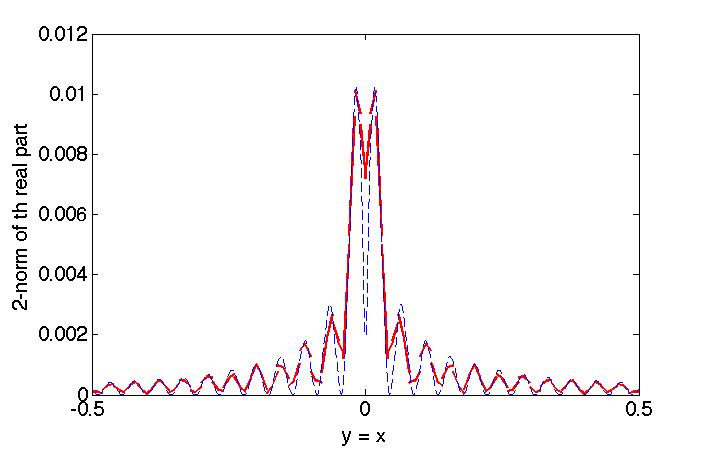} \includegraphics[scale = .25]{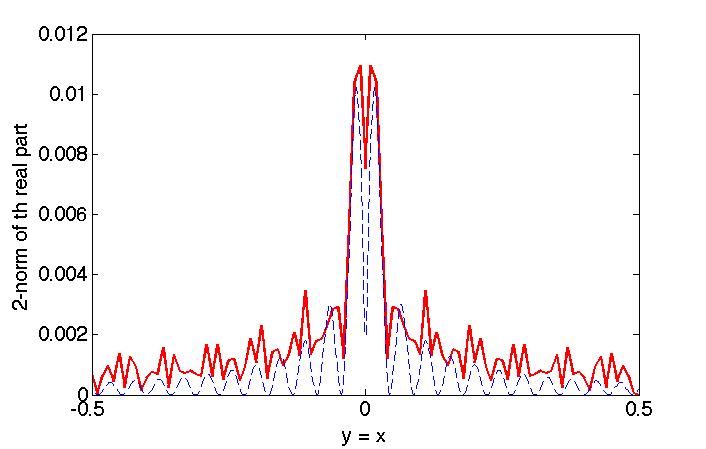}}
\caption{The left plot is of $\| \re (\bfu_h) \|_{L^2(\Ome)}$ (solid red line) vs. 
$\| \re (\bfu) \|_{L^2(\Ome)}$ (dashed blue line) for $h = 1/50$.  The right plot is of 
$\| \re (\bfu^{FEM}_h) \|_{L^2(\Ome)}$ (solid red line) vs. $\| \re (\bfu) \|_{L^2(\Ome)}$ 
(dashed blue line) for $h = 1/50$. \label{fig:IP-DGVSFEM1}}
\end{figure}

\begin{figure}[htb]
\centerline{\includegraphics[scale = .25]{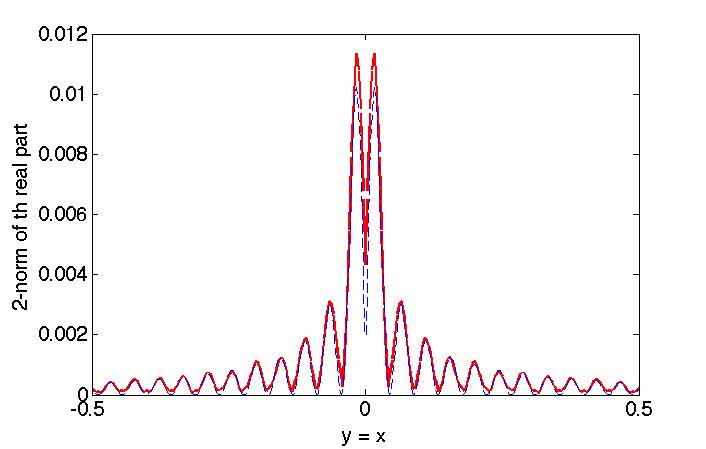} \includegraphics[scale = .25]{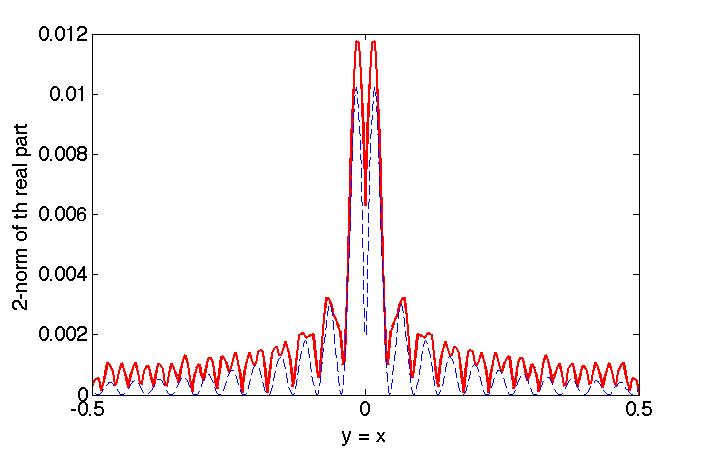}}
\caption{The left plot is of $\| \re (\bfu_h) \|_{L^2(\Ome)}$ (solid red line) vs. 
$\| \re (\bfu) \|_{L^2(\Ome)}$ (dashed blue line) for $h = 1/120$.  The right plot is
of $\| \re (\bfu^{FEM}_h) \|_{L^2(\Ome)}$ (solid red line) vs. $\| \re (\bfu) \|_{L^2(\Ome)}$ 
(dashed blue line) for $h = 1/120$. \label{fig:IP-DGVSFEM2}}
\end{figure}

\begin{figure}[htb]
\centerline{\includegraphics[scale = .25]{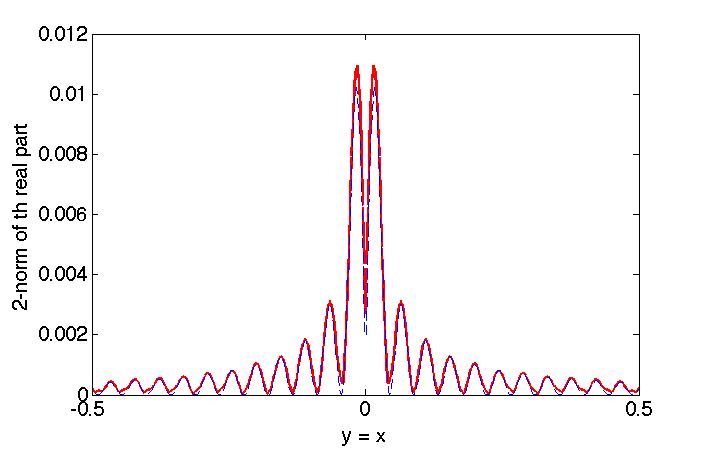} \includegraphics[scale = .25]{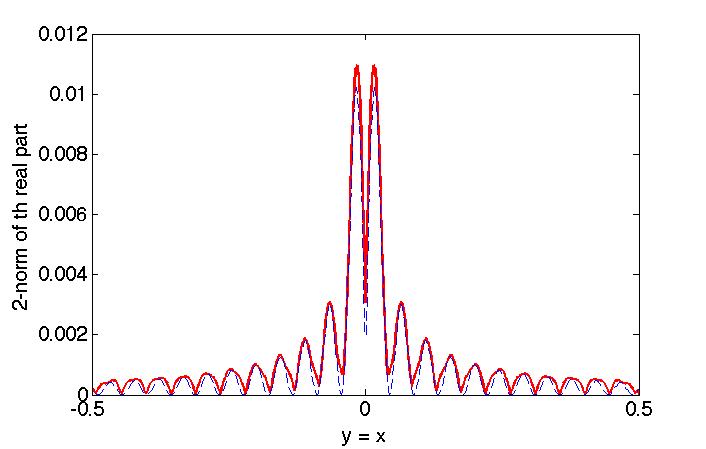}}
\caption{The left plot is of $\| \re (\bfu_h) \|_{L^2(\Ome)}$ (solid red line) vs. 
$\| \re (\bfu) \|_{L^2(\Ome)}$ (dashed blue line) for $h = 1/200$.  The right plot is 
of $\| \re (\bfu^{FEM}_h) \|_{L^2(\Ome)}$ (solid red line) vs. $\| \re (\bfu) \|_{L^2(\Ome)}$ 
(dashed blue line) for $h = 1/200$. \label{fig:IP-DGVSFEM3}}
\end{figure}

%%%%%%%%%%%%%%%%%%%%%%%%%%%%%%

\end{document}